\documentclass[final]{siamltex}

\usepackage{epsf}
\usepackage{graphicx}% Extended graphics
\usepackage{amsmath} \usepackage{amssymb}
\usepackage[titletoc]{appendix}

\newtheorem{example}{Example}[section]

\newtheorem{remark}{Remark}[section]

\def\a{\alpha} \def\b{\beta}

\def\dref#1{(\ref{#1})}

 \def\sin{\mbox{sin}\,}

 \def\dfrac{\displaystyle\frac}
\def\be{\begin{equation}}
\def\bel{\begin{equation}\label}
\def\ee{\end{equation}}
\def\ba{\begin{array}}
\def\ea{\end{array}}
\def\banl{\begin{eqnarray}\label}
\def\ean{\end{eqnarray}}

 \def\bna{\begin{eqnarray}}
\def\ena{\end{eqnarray}} \def\dref#1{(\ref{#1})}

\title{Identification of Disturbed Control Systems\thanks{This work was supported in part by the National Natural Science Foundation of China under Grants 61422308 and 11688101.}}

\author{Chanying Li\thanks{C.~Li is with  the Key Laboratory of Systems and Control, Academy of Mathematics and Systems Science, Chinese Academy of Sciences, Beijing 100190, P.~R.~China and the School of Mathematical Sciences,  University of Chinese Academy of Sciences, Beijing 100049, P.~R.~China.}}

\begin{document}
\maketitle

\begin{abstract}
This paper studies the identification of nonlinearly parameterized control systems in  given experiments. Several identifiability criteria are established and an implementable algorithm  is proposed for practicality  with the convergence rate explicitly computed.
\end{abstract}



\pagestyle{myheadings}
\thispagestyle{plain}
\markboth{ }{ }

\section{Introduction}
Consider the nonlinearly parameterized control system
\begin{equation}\label{hmodel}
\left\{
\begin{array}{l}
x_{t}=f(\theta,u_{t-1},\chi_{t-1})+w_{t}\\
y_t=h(\theta,x_t)+v_t
\end{array}
,\quad t\geq 1,
\right.
\end{equation}
where  $x_t,  y_t,   u_t$ and $(w_t, v_t) $ represent   the  $p\times 1$ state vector,  $q\times 1$ output  vector,  $r\times 1$ input vector and $(p+q)\times 1$  noise vector, respectively. Denote
$\chi_t\triangleq (x_t,\ldots, x_{t-m+1})$ as the state regressor.  %with  initial vector
%$\chi_0=(x_0,\ldots,x_{-m+1}) $ %non-random.  being independent of noises $\{w_t\}$ and $\{v_t\}$.
Unknown parameter $\theta$ is non-random and belongs  to a  known nondegenerate compact hyperrectangle $\Theta\subset \mathbb{R}^n$. Moreover,
$f: \mathbb{R}^n\times \mathbb{R}^{r}\times \mathbb{R}^{pm}\to \mathbb{R}^p$ and   $h: \mathbb{R}^n\times \mathbb{R}^{p}\to \mathbb{R}^q$ are two known %Lebesgue measurable
functions. Let
 $h^{-1}:  \mathbb{R}^n\times  \mathbb{R}^q\rightarrow 2^{\mathbb{R}^p}$     be a set-valued function that
 $h^{-1}(x,y)\triangleq \{z: h(x, z)=y\}$, then assume
 \begin{description}
\item[A1]
The noises $\{w_t\}$ and $\{v_t\}$ are two i.i.d sequences  satisfying:\\ % with  $\mathcal{F}^w_\infty$ and  $\mathcal{F}^v_\infty$  independent. \\
(i)   $\{w_t\}$ is independent of $\{v_t\}$;\\
(ii) for each $t\geq 1$,   $(w_t,v_t)$  is independent of  $\chi_0$ and $ \{u_i\}_{0\leq i\leq t-1}$;\\
(iii)  $\|w_1\|\leq C_w$ and $\|v_1\|\leq C_v$ for  some $C_w>0$ and $ C_v\geq 0$. In addition,
   \begin{equation}\label{WV>0}
   \inf\nolimits_{z\in \mathcal{W}\times \mathcal{V} } P( (w_1,v_1)\in    B(z, \delta))>0,\quad \forall \delta>0.
   \end{equation}
   where   $  \mathcal{W}\triangleq\overline {B(0,C_w)}\subset \mathbb{R}^{p}$ and $ \mathcal{V}\triangleq\overline {B(0,C_v)}\subset \mathbb{R}^{q}$.
\item[A2]  $f$ and $h$ are continuous;  $h^{-1}$ is  bounded-valued and
upper semicontinuous\footnote{ A set-valued function $\zeta:X\rightarrow 2^Y$  is bounded-valued if for $\forall x\in X$, $\zeta(x)$ is bounded. $\zeta$
is said to be upper semicontinuous if for any $x\in X$ with   $\zeta(x)\neq \emptyset$ and any neighborhood $U$ of $\zeta(x)$, there is a $d_x>0$ such that $\zeta(B(x,d_x))\subset U$.   }.
\end{description}

An important issue in system identification is to solve the  identifiability of system \dref{hmodel} in an
 experiment  $(\chi_0, \{u_t\})\in \mathcal{E}$, where $\mathcal{E}$ is the set of all admissible experiments defined by  %that independent of $\{w_t,v_t\}$. %
\begin{equation}\label{E}
 \mathcal{E}\triangleq \{(\chi_0, \{u_t\}_{t\geq 0}): \|u_t\|\leq C_u, t\geq 1 \}\quad\mbox{for}\quad C_u>0.
\end{equation}
%$$\mathcal{E}\triangleq \{(\chi_0, \{u_t\}_{t\geq 0}): u_t\in \mathcal{F}^y_t \triangleq \sigma\{ \varphi_i, 0\leq i\leq  t ), t\geq 0 \},$$
%where $\varphi_t\triangleq (y_t, \ldots, y_{t-m+1})$ and $\varphi_0$ is non-random.
% But we focus on a familiar scenario in this paper where the  input of system \dref{hmodel} at any instance  depends only on a finite number of past observations, without loss of generality, say $m$ observations. Since   system \dref{hmodel} is time-invariant, we are interested in
%\begin{equation}\label{utxi}
%u_t=\xi(\varphi_t),\quad t\geq 0,
%\end{equation}
%where $\xi:\mathbb{R}^{qm}\rightarrow \mathbb{R}^r$ is  continuous and   initial $\varphi_0$ is non-random.
%%So, the   admissible experiment set  $\mathcal{E}_0\subset\mathcal{E}$  becomes
%\begin{equation}\label{adE}
%\mathcal{E}_0\triangleq \{(\chi_0, \{u_t\}_{t\geq 0})\in \mathcal{E}:  u_t\,\,\mbox{satisfies \dref{utxi}} %for each some Lebesgue measuable }\xi,\,\,t\geq 0
%\}.
%\end{equation}
This direction arises from numerous engineering applications where identification has to be performed in  control processes, especially with  feedbacks inherent  \cite{AW71}, \cite{FL99}, \cite{gevers86}, \cite{GBP74}, \cite{GLSlover77}, \cite{HS95}, \cite{lichen16}. Unlike  identification operating in open loop,  a prominent feature of closed-loop identification is that there is no design   level on data in parameter estimation, once a feedback law is chosen. In this paper, we assume that the experiment is designed  in advance for control purposes. Then, outputs $y_t$ will be produced by control system \dref{hmodel} automatically.  We aim to identify parameter $\theta$  in the running process of the control system.

 Historically, identification  of noise-free systems from input-output data has been well addressed. Literatures on this topic have also shed some light on the determining factor of identifiablity for disturbed control systems.   As stated by \cite{glover76},   parameter identification is  in nature a  procedure of distinguishing output trajectories of different parameters. From this viewpoint, the critical criterion, in some sense, on linear system structure was  deduced by \cite{glover76}. Nonlinear systems with noises absent were  treated therein as well.  Considering noises, however, different observations might be produced by  the same parameter. We thus introduce the definition of identifiability for  disturbed control systems as following.

\begin{definition}\label{thstru}
System \dref{hmodel} is  identifiable under  experiment  $(\chi_0, \{u_t\})\in \mathcal{E}$,   if there is an estimator   such that the unknown parameter  $\theta$ in $\Theta$   can be uniquely determined by  the data set    $Z^\infty\triangleq \{y_{t+1}, u_t \}_{ t\geq 0}$   with probability $1$.
\end{definition}

Examining output trajectories to check  identifiability    is not straightforward  in  most  circumstances. So, interesting move to  derive  some simple identifiability criteria. This is exactly  the first part of the paper, where it is argued in Section \ref{CLI}   that the excitation points of  control system \dref{hmodel} are crucial for identifiability. In fact,  given  any experiment in $\mathcal{E}$, the identifiability of   system \dref{hmodel} is ensured
if   the excitation point set is  sufficiently dense. A lower bound of the required   density %of thetion points
is computed accordingly. On the other hand, if the density of the excitation points is smaller  than the lower bound, the  identification may possibly fail. Generally speaking, this structure condition   for identifiability   is  weaker than that for the noise-free case. This is because noises $\{w_t\}$ in the state equation  are advantageous
 in identification, as suggested by the results.

 Since the estimator studied for the proofs of Theorems \ref{idensta} and \ref{iden}  is only of theoretical interest,
 the second part of the paper is intended to introduce an implementable algorithm
for the sake of practicality. The proposed  estimator  is called the  grid searching (GS) estimator and has its origins  in the nonlinear least-squares (NLS) method, whose asymptotic behaviours  and  approximation algorithms  have been explored for decades  \cite{GM1978}, \cite{GLT},  \cite{Jabob}, \cite{lai}, \cite{Wu81}. % Despite of these existing    results, there still remains some distance away from solving the closed-loop identification problem.
By modifying the NLS method in Section \ref{ImA},
the GS estimator  is proved to be strong consistent for a basic class of disturbed control systems    under  some appropriate conditions. % whenever system states are observed.
This estimator
 can also cope with the situation where the noise variances are unknown.
%The convergence rates   are provided as well.

%The rest of the paper is organized as follows.  Section \ref{CLI} addresses  the issue of identifiability for control system  \dref{hmodel} by considering the C-recurrent case and the general case separately. Section \ref{ImA} studies the  GS estimator with the convergence rate explicitly computed.

\section{Identifiability for Control Systems}\label{CLI} We shall establish some  identifiability criteria for system \dref{hmodel} on the basis of experiment data.
\subsection{Notations}
Throughout this paper, we consider the probability measure space $(\Omega, \mathcal{F}, P)$. The  notations and definitions  used in this section are introduced here.  Let          diam$(x, A)\triangleq \sup_{x'\in A} d(x,x')$, where $d(\cdot,\cdot)$ denotes the distance between two points.
%Denote $B_d\triangleq B(0,d)$ and $\bar B_d\triangleq \overline{ B(0,d)}, \forall d>0$.
Denote $ \varphi \triangleq (z_1,\ldots,z_m)$, $\psi \triangleq (z'_1,\ldots,z'_m)$ and  $\chi \triangleq (z''_1,\ldots,z''_m)$  with $z_j, z'_j\in \mathbb{R}^q, z''_j\in \mathbb{R}^p,  j\in [1,m]$. % Let $\xi:\mathbb{R}^{qm}\rightarrow \mathbb{R}^r$, then
 Then, for $x\in \mathbb{R}^n,  w\in \mathbb{R}^p, v\in  \mathbb{R}^q$,  define
\begin{eqnarray}\label{barfh}
\left\{
\begin{array}{l}
\bar{h}(x,\chi,\psi)\triangleq(h(x,z''_1)+z_1',\ldots,h(x,z''_m)+z'_m)\\
%h'_m(x,\chi,\psi)\triangleq h_m(x,\chi)+\psi\\
%h^{-1}_m(x,\varphi)\triangleq(h^{-1}(x,z_1),\ldots,h^{-1}(x,z_m))\\
%f'(x,\chi,\psi)\triangleq f(x,\xi(h_m(x,\chi,\psi)), \chi )\\
\bar h^{-1}(x, \varphi,  \psi  )\triangleq  (h^{-1}(x,z_1-z'_1), \ldots, h^{-1}(x,z_m-z'_m))\\
 h'(x,u,\chi,w,v)\triangleq h(x, f(x,u,\chi)+ w) +v\\
\hat  h(x, u, \varphi,  \psi, w, v)\triangleq \bigcup_{\chi\in \bar h^{-1}(x, \varphi,  \psi  )}   h'(x,u,\chi,w,v)
\end{array}.
\right.
\end{eqnarray}
So, $%\tilde{h}^{-1},
\bar h^{-1}$ and $\hat h$ are set-valued functions. % Let $  \mathcal{W}\triangleq\overline {B(0,C_w)}\in \mathbb{R}^{p}$ and $ \mathcal{V}\triangleq\overline {B(0,C_v)}\in \mathbb{R}^{q}$ for some $C_w>0$ and $ C_v\geq 0$.
Denote the images of  $\bar h$, $ h'$  and  $\hat{h}$  at fixed points  $(x,\chi)$, $(x,u, \chi)$
 and $(x,u,\varphi)$, respectively,
    by $\mbox{Im} (\bar{h}_{x,\chi})\triangleq \bar{h}(x,\chi,\mathcal{V}^m)$,
%\begin{eqnarray*}
%\left\{
%\begin{array}{l}
$$
\mbox{Im}( h'_{x, u,  \chi})\triangleq h'\left(x, u, \chi,  \mathcal{W}, \mathcal{V}\right )  \quad\mbox{and}\quad  \mbox{Im}(\hat h_{x, u, \varphi})\triangleq \hat  h\left(x,u,\varphi, \mathcal{V}^m,  \mathcal{W},\mathcal{V}\right ),$$
%\end{array}.
%\right.
%\end{eqnarray*}
%where $  \mathcal{W}= \overline {B(0,C_w)}\in \mathbb{R}^{mp}$ and $ \mathcal{V}=\overline {B(0,C_v)}\in \mathbb{R}^{mq}$ with $C_w>0$ and $ C_v\geq 0$.
Now, let $k,l\in \mathbb{N}^+$.  View  set $Z\subset \mathbb{R}^{k}$  as a point  $\check{z}\in 2^{\mathbb{R}^k}$    and by a slight abuse of notation, we write $\check{z}=Z$.
%Given  $\epsilon>0$, for any  $z\in \mathbb{R}^k$, denote  $\check{z}_\epsilon=B(z, \epsilon)\subset 2^{\mathbb{R}^k}.$
%let $\check{z}_i,i\in [1,s]$ denote the sets $Z_i\in 2^{\mathbb{R}^{k_i}}$.
%We say $\check{z}_1\neq \check{z}_2$ for $ \check{z}_i\in \mathbb{R}^k, i=1,2$, if $B(z_1,\epsilon)\cap B(z_2,\epsilon)=\emptyset$.
Now, for function (respectively, set-valued function)  $\zeta: \mathbb{R}^{k} \rightarrow \mathbb{R}^l$ (respectively, $2^{\mathbb{R}^l}$), define $\mathcal{B}\zeta: 2^{\mathbb{R}^{k}} \rightarrow 2^{\mathbb{R}^l}$ by $ \mathcal{B}\zeta(\check{z})=\zeta\left( Z\right).$
%$\zeta: \prod_{i=1}^s\mathbb{R}^{k_i} \rightarrow \mathbb{R}^l$ (or $2^{\mathbb{R}^l}$),
%define $\mathcal{B}\zeta: 2^{\prod_{i=1}^s\mathbb{R}^{k_i}} \rightarrow 2^{\mathbb{R}^l}$:
%\begin{equation}\label{disturbance}
%\mathcal{B}\zeta(\check{z})\triangleq\zeta\left( Z\right).
%\end{equation}
Let $ \zeta_i$ be  two functions and $\check{z}_i=Z_i, i=1,2$.   We say
\begin{equation}\label{BneqB}
\mathcal{B}\zeta_1(\check{z}_1) \neq \mathcal{B}\zeta_2(\check{z}_2)\quad \mbox{if}\quad \zeta_1\left( Z_1\right) \cap \zeta_2\left( Z_2\right)=\emptyset.
\end{equation}
Given $\epsilon>0$, for any $z\in \mathbb{R}^k$,   denote $\check{z}_\epsilon= B(z, \epsilon)\in 2^{\mathbb{R}^k}$. % and  $\check{z}_0=\{z\}$.
  Then,
let $ \mathcal{\check{V}}^m_\epsilon\triangleq \bigcup_{\psi \in \mathcal{V}^m }\check{\psi}_\epsilon$ and %$\mathcal{\check{W}}=\bigcup_{w\in \mathcal{W}} \check{w}_\varepsilon$
$\check{\Pi}_\epsilon\triangleq\bigcup_{\pi\in \mathcal{W}\times \mathcal{V}} \check{\pi}_\epsilon$.
Define the images of $\mathcal{B}\bar h$,    $\mathcal{B} h'$ and   $\mathcal{B}\hat h$     at points  $(\check{x},\check{\chi})\in 2^{\mathbb{R}^n}\times 2^{\mathbb{R}^{pm}}$,  $(\check{x},\check{u},\check{\chi})\in 2^{\mathbb{R}^n}\times 2^{\mathbb{R}^{r}}\times 2^{\mathbb{R}^{pm}}$
 and   $(\check{x},  \check{u},  \check{\varphi})\in 2^{\mathbb{R}^n}\times 2^{\mathbb{R}^r}\times2^{\mathbb{R}^{qm}}$  by
\begin{eqnarray}\label{Bhhh}
\left\{
\begin{array}{l}
\mbox{Im}(\mathcal{B}\bar h)^\epsilon_{\check{x},\check{\chi}  }\triangleq \mathcal{B}\bar  h\left(\check{x},\check{\chi}, \mathcal{\check{V}}^m_\epsilon\right ),\quad
\mbox{Im}(\mathcal{B} h')^\epsilon_{\check{x},\check{u},\check{\chi}}\triangleq \mathcal{B} h'\left(\check{x},\check{u},\check{\chi},  \check{\Pi}_\epsilon\right )\\
\mbox{Im}(\mathcal{B}\hat h)^\epsilon_{\check{x},  \check{u}, \check{\varphi}  }\triangleq \mathcal{B}\hat  h\left(\check{x},\check{u},\check{\varphi}, \mathcal{\check{V}}^m_\epsilon, \check{\Pi}_\epsilon\right )
\end{array}.
\right.
\end{eqnarray}
%$$
%\mbox{Im}(\mathcal{B}\bar h)^\epsilon_{\check{x},\check{\chi}  }\triangleq \mathcal{B}\bar  h\left(\check{x},\check{\chi}, \mathcal{\check{V}}^m_\epsilon\right ),\quad
%\mbox{Im}(\mathcal{B}\hat h)^\epsilon_{\check{x},  \check{u}, \check{\varphi}  }\triangleq \mathcal{B}\bar  h\left(\check{x},\check{u},\check{\varphi}, \mathcal{\check{V}}^m_\epsilon,  \mathcal{\check{R}}_\epsilon\right )\quad\mbox{and}\quad
%\mbox{Im}(\mathcal{B} h')^\epsilon_{\check{x},\check{u},\check{\chi}}\triangleq \mathcal{B} h'\left(\check{x},\check{u},\check{\chi},  \mathcal{\check{R}}_\epsilon\right ).$$

%So, the disturbance function of  $\mbox{Im}(\bar h_{x,  \varphi})$ is
%\begin{equation}\label{disturbanceIm}
%\mathcal{D}(\mbox{Im}(\bar h_{x,  \varphi}),\epsilon)=  \bigcup_{(x',\varphi')\in B((x,\varphi),\epsilon)} \bar  h\left(x',\varphi', \mathcal{V}^m,  \mathcal{W}, \mathcal{V}\right )
%\end{equation}

%\subsection{ Identifiability Criteria}
\subsection{Motivations and Excitation Points}
Let us first look at a simple  system
\begin{equation}\label{model}
y_{t}=f(\theta,\varphi_{t-1})+f'(u_{t-1},\varphi_{t-1})+w_{t},\quad t\geq 1,
\end{equation}
 where $\varphi_t\triangleq (y_t, \ldots, y_{t-m+1})$  is an observable $pm\times 1$ vector. % and % $f: \mathbb{R}^n\times \mathbb{R}^{mp}\rightarrow \mathbb{R}^p$  and   $f': \mathbb{R}^r\times \mathbb{R}^{mp}\rightarrow \mathbb{R}^p$
% $f, f'$ are two  known  Borel  measurable functions.
The experiment thus becomes $(\varphi_0,\{u_t\})$ in $\mathcal{E}$
   and    % $\varphi_0$ is non-random.
 Assumptions A1--A2 degenerate to
\begin{description}
\item[A1']
$\{w_t\} $ is an  i.i.d sequence satisfying\\
(i) for each $t\geq 1$,   $w_t$  is independent of  $\chi_0 $ and $\{u_i\}_{0\leq i\leq t-1}$;\\
(ii)  $\|w_1\|\leq C_w$  for  some finite $C_w>0$ and
  \begin{equation}\label{w1Cw}
   \inf_{z\in \mathcal{W} %\overline {B(0,C_w)}
   } P( w_1\in    B(z, \delta))>0,\quad \forall \delta>0.
   \end{equation}
   \item[A2'] $f$ and $f'$ are continuous.
\end{description}

The most familiar experiments  are the ones that  casue $\|\varphi_t\|\leq C$,  $\forall t\geq 1$  almost surely for some $C>0$. %That is, the experiments are stable.
Apparently, if $\|\varphi_t\|\leq C$, by   \dref{hmodel}, \dref{E} and Assumption A2', it is easy to compute a $ C_0>0$ that
  \begin{equation}\label{ff'<}
f_i\triangleq\|f(\theta,\varphi_{i})+f'(u_{i},\varphi_{i})\|\leq  C_0,\quad i=t,\ldots,t+m-1,
   \end{equation}
and hence
 %\begin{equation}\label{S}
 $(f_{t},\ldots,f_{t+m-1})\in\mathcal{S}\triangleq \prod_{i=1}^m \overline {B(0,C_0)}\subset \mathbb{R}^{pm}$. So, the following result is not suprising.
\begin{theorem}\label{idenf}
Under Assumptions A1'--A2', let  $ (\varphi_0, \{u_t\})\in \mathcal{E}$ be an experiment such that $P\{\|\varphi_t\|\leq C, i.o.\}=1$ for some $C>0$.
Then, control
system \dref{model}  is identifiable  if  for each pair $x,x'\in \Theta$ with $ x\neq x'$, there are  sufficiently dense points $\b \in \mathcal{ S}$ such that
 $ f(x,\b)\neq f(x',\b)$.
\end{theorem}

This theorem is a direct consequence of Theorem \ref{idensta} appearing in a later section.
% produced by an experiment $(\chi_0, \{u_t\})\in \mathcal{E}$. %Start with the problem by stating two elementary assumptions  below:
The observation of the above theorem  %Theorem \ref{idenf}
enlightens us to  introduce   set
\begin{equation}\label{macalP}
\mathcal{P}_\a\triangleq\{\eta\in \mathbb{R}^{pm}: \exists \b\in \mbox{Im} (\bar{h}_{x,\eta} ) \mbox{  s.t.  Im}(   h'_{x,u,  \eta}  ) \nsubseteq \mbox{Im}(\hat h_{x', u,  \b})\},
\end{equation}
where $u$ is restricted to $\overline{B(0,C_u)  }$ and
\begin{equation}\label{alpha}
\a \in A_0 \triangleq\{(x,x')\in \Theta\times \Theta: x\neq x'\}.
\end{equation}
%the  points $(\b,\eta)$ satisfying \dref{excitationsta}
%The density of $P_\a$  is crucial to the identifiability of system \dref{hmodel} in an experiment $(\chi_0, \{u_t\})\in \mathcal{E}$.
 We call $\eta\in \mathcal{P}_\a$  an \emph{excitation point} of $\a\in A_0$ for system \dref{hmodel}.  If a system $(f,h)$ has  sufficiently  dense  excitation points of $\a$, then  states $\chi_t$ are very likely to fall in $ \mathcal{P}_\a$. This means it is  relatively easy to   distinguish $x$ and $x'$.
 \begin{example}
Consider  system \dref{model},
in which case $\eta=\b$ and
$$
 \mbox{Im}(   h'_{x,u, \eta}  ) \nsubseteq \mbox{Im}(\hat h_{x', u, \b})\quad \Leftrightarrow \quad f(x,\b)\neq f(x',\b).
$$
Heuristically,   $\mathcal{P}_\a=\{\b: f(x,\b)\neq f(x',\b)\}$ is composed of the points where different parameters give rise to different values of $f$.
\end{example}

Theorem \ref{idenf} suggests that the identifiability of a control system depends on   the density of $P_\a,\a\in A_0$.
More precisely, for two sets $Z, Z'\in \mathbb{R}^l, l\geq 1$, we define the \emph{lower density} of $Z'$ in $Z$  by
\begin{equation*}\label{density}
\underline d(Z'|Z)=\dfrac{1}{\sup_{z\in Z}\inf\{d>0: B(z,d)\cap Z'\neq \emptyset           \} }.
\end{equation*}
%Clearly, if $Z'$ is dense in $Z$, then $\underline d(Z'|Z)=\infty.$
Further, when $Z=\prod_{j=1}^m Z_j, Z_j\subset \mathbb{R}^{l}, l\geq 1$ and $Z'\subset \mathbb{R}^{lm}$, the \emph{$m$-symmetric  lower density} of $Z'$  in $Z$ is defined by
 $$
 \underline d^m(Z'|Z)\triangleq \sup_{\prod\nolimits_{j=1}^m E_j\subset Z', E_j\in \mathbb{R}^l}(\min\nolimits_{j\in [1,m] } \underline d(E_j|Z_j)).
 $$
 %if there is a set $\prod_{j=1}^m E_j\subset Z'$ with $ E_j\subset \mathbb{R}^{k}$ and $\underline d(E_j|Z_j)=\rho, j\in [1,m].$
 Clearly, $\underline d^1(Z'|Z)=\underline d(Z'|Z)$.  To identify parameter $\theta$, the density of  $\mathcal{P}_\a$ for  control  system \dref{hmodel} % in an experiment $(\chi_0, \{u_t\})\in \mathcal{E}$
  is deduced in the next subsection.

\subsection{ Identifiability Criteria}
The criteria are presented in two cases.

\subsubsection{Criterion for $C$-Recurrence}\label{Crecurrency}
System states  are usually  constrained in a bounded area
 in practice. It is a special case of $C$-recurrence defined below:
\begin{definition}
 An  experiment   $(\chi_0, \{u_t\})\in \mathcal{E}$ is said to be $C$-recurrent for some  $C>0$, if    the corresponding states  %of system \dref{hmodel}
  satisfy
$P\{\|\chi_t\|\leq C, i.o.\}=1.$ \\
%It is said to be $C$-divergent, if $P\{\|\chi_t\|\geq C, i.o.\}=1.$\\
\end{definition}
%Clearly, an experiment is either $C$-recurrent or $C$-divergent for a given $C>0$. We first study the former situation
%When $(\chi_0, \{u_t\})\in \mathcal{E}$, then there is a $C_f>0$ such that for all $t\geq m$ with $\|\chi_{t-m}\|\leq C$,
%\begin{equation}\label{Cf}
%\|f(\theta, u_{t-j}, \chi_{t-j})\|\leq C_f,\quad 1\leq j\leq m.
%\end{equation}
%For a $C$-recurrent experiment with $C>0$
% finite,  the desired density of  $\mathcal{P}_\a$ in  $\mathcal{S}$  is now computed  for identifiability.      % is required to be sufficiently dense in $\mathcal{S}$.%, where
 %\begin{equation}\label{S}
% $\mathcal{S}= \prod_{i=1}^m \overline {B(0,C+C_w)}\subset \mathbb{R}^{mp}$.
%\end{equation}
%and $C_i$ is calculated   by  $C_0=C$,
%\begin{equation}\label{Cf}
%\left\{
%\begin{array}{ll}
%C_1\triangleq \max_{x\in\Theta,\|\chi\|\leq C, \psi\in \mathcal{V}^m}\|f'(x,\chi,\psi)\| \\
%C_j\triangleq \max_{x\in\Theta,\|\chi\|\leq \sum_{l=0}^{j-1}C_l+(j-1)C_w,  \psi\in \mathcal{V}^m}\|f'(x,\chi, \psi)\|,\quad j\in[1,m].
%\end{array}.
%\right.
%\end{equation}
%Here,  $f'$ is defined by \dref{barfh}.
The main result of this section is stated as follows.
%with $C_f\triangleq \max_{(x,\chi,\psi)\in\Theta\times \overline {B(0,C)}\times \mathcal{V}^m} \|f(x, \xi( h_m(x,\chi,\psi) ), \chi)\|.$
\begin{theorem}\label{idensta}
%Let  $ (\chi_0, \{u_t\})\in \mathcal{E}$ be $C$-stable.
Under Assumptions A1--A2,    control  system \dref{hmodel}
 is identifiable  for any $C$-recurrent experiment  $(\chi_0, \{u_t\})\in \mathcal{E}$  if    $\underline d^m (\mathcal{P}_\a| \mathcal{S})> 1/C_w$ for each $\a\in A_0$.
\end{theorem}

\begin{remark}\label{tight}
General speaking, the lower bound $1/C_w$ in Theorem \ref{idensta}  cannot be further relaxed.
%If  $\underline d^m (\mathcal{P}_\a| \mathcal{S})\leq 1/C_w$ for some $\a\in A_0$, then  the identifiability of control system \dref{hmodel}  may fail.
For  example, %in Theorem \ref{idensta},
consider system \dref{model} with
 $[1,2]\subset \Theta\subset \mathbb{R}$, % $\varphi_{t}=(y_t,y_{t-1})$, where  $y_t$ is a scalar (
 $m=2, p=1$. Assume  $f'(u,y_1,y_2)=uy_2$ and
 % $f$ satisfies %the following:
\begin{eqnarray*}
 f(x,y_1,y_2)=\left\{
\begin{array}{ll}
0, & y_1\in [-C_w,C_w]\\
x(y_1-C_w), & y_1>C_w\\
x(y_1+C_w), & y_1<-C_w
\end{array}, \quad x\in [1,2].
\right.
\end{eqnarray*}
%(i) there is a $C>0$ such that $C_2\leq C/\sqrt{2}$;\\
%$$
%\max_{x\in\Theta,\|\varphi\|\leq C+C_l+C_w} \bar  f(x,\varphi)
%$$
%(ii) for  some  $(x_1,x_2)\in A_0$ and $\beta\in \overline{B(0,C_2)}\subset \mathbb{R}^p$,
%\begin{eqnarray*}
%\left\{
%\begin{array}{ll}
%\bar  f(x_1,\varphi)\equiv \bar  f(x_2,\varphi)\equiv \b, & \varphi\in  \overline{B(\b,C_w)}  \times \mathbb{R}^p\\
%\bar  f(x,\varphi)\neq \bar f(x',\varphi), & (x,x')\times \varphi\notin (x_1,x_2) \times   \overline{B(\b,C_w)}\times \mathbb{R}^p
%\end{array}.
%\right.
%\end{eqnarray*}
%$f(x_1,y)\equiv f(x_2,y)\equiv \b$ for $y\in [\b-C_w,\b+C_w]$.
%Take parameter $\theta=x_1$.
It is evident that $\underline d^2 (\mathcal{P}_{(1,2)}| \mathcal{S})= 1/C_w$. Moreover,
 $\theta$ cannot be identified in experiment
 $((0,0), \{0\})$, % with $\varphi_0\in \mathbb{R}^p\times \overline{B(\b,C_w)}$.
which is    $C$-recurrent for any given $C>0$.
\end{remark}

%Theorem \ref{idensta} implies that in  a $C$-recurrent experiment $ (\chi_0, \{u_t\})\in \mathcal{E}$ where  system \dref{hmodel} is identifiable,  $\mathcal{ P}_\a, \a\in A_0$ is possible to be very sparse   that contains only one   point $\eta_\a$. This is because $\underline d^m (\mathcal{P}_\a| \mathcal{S})> 1/C_w, \a\in A_0$  whenever
%\begin{equation}\label{CwCm}
%C_w>C_m+\sup\nolimits_{\a\in A_0}\| \eta_\a\|.
%\end{equation}
%We point out that $\{\eta_\a, \a\in A_0\}$ can be taken as a finite set $\{\eta_i\}_{ 1\leq i\leq N}$ for some integer $N>0$ in \dref{CwCm}, as proved by  Lemma \ref{CoverA} later. % So, a large but finite $C_w$ is sufficient for   \dref{CwCm}.
%In particular, for system
%\begin{eqnarray}\label{fyu}
%y_{t}=f(\theta, u_{t-1}, \varphi_{t-1})+w_{t}%=\bar f(\theta,\varphi_{t-1})+w_{t}
%,\quad t\geq 1%,\quad y_t,u_t,w_t \in \mathbb{R}.
%\end{eqnarray}
 %with $m=1$, it is clear that $C_m$ and $\{\eta_i\}_{ 1\leq i\leq N}$ are both independent of $C_w$. So, \dref{CwCm} holds whenever $C_w$ is sufficiently large.

%\begin{corollary}\label{idenCw}
%Consider control
% system    \dref{fyu} with  $m=1$ and $\mathcal{ P}_\a \neq \emptyset$ for each $ \a\in A_0$.
%Under Assumptions A1--A2, if the underlying experiment
%$(\chi_0, \{u_t\})\in \mathcal{E}$ is  $C$-recurrent,  then
% system    \dref{fyu}
% is identifiable  for all sufficiently large $C_w$.
%\end{corollary}

\begin{remark}
%Corollary \ref{idenCw} indicates that % the parameter  is  easier to be identified for a larger $C_w$. So,
To some extent,  noises   $\{w_t\}$ in the state equation are advantageous
 in the closed-loop identification, whereas $\{v_t\}$ in the observation equation   play an opposite role. This observation becomes clear during the proof of Theorem \ref{idensta}.
\end{remark}

%\begin{remark}
%A typical example for Corollary \ref{idenCw} is system
%\end{remark}

%\begin{remark}
%The coming proof suggests that for $ \a=(x,x')$  with $x\neq x'$,
%$$
%\mathcal{P}_\a\triangleq\{( \b, \eta): \mbox{Im}(   h'_{x, \eta}  ) \nsubseteq \mbox{Im}(\bar h_{x',  \b})\}
%$$
%the  points $(\b,\eta)$ satisfying \dref{excitationsta}
% is crucial to the identifiability of system \dref{hmodel}. A $(\b,\eta)\in \mathcal{P}_\a$  is called an \emph{excitation point} of $\a$. If a system $(f,h)$ has a sufficiently  dense  excitation point set $\mathcal{P}_\a$ for some $\a$, it is then relatively easy to   distinguish $x$ and $x'$.
%\end{remark}

\subsubsection{Criterion for  General Case}
Generally, given an $\a \in A_0$,
the excitation points
 of $\a$  are expected in the following set for some $\epsilon>0$: % ($\epsilon$ depends on $\a$):
 %to fulfill the following relationship for some $\varepsilon>\epsilon>0$:
% \begin{eqnarray}\label{uthege}
%\left\{
%\begin{array}{l}
%h_m(\check{x}_\epsilon,\check{\eta}_\epsilon,\check{0}_\epsilon) \subset  \check{\b}_\varepsilon\\
%\mbox{Im}(\mathcal{B}  h')_{\check{x}_\epsilon,\check{\b}_\varepsilon, \check{\eta}_\epsilon} \nsubseteq  \mbox{Im}( \mathcal{B}\bar h)_{\check{x'}_\epsilon, \check{\b}_\varepsilon}
%\end{array},
%\right.
%%where  $\check{x}_\epsilon=B(x,\epsilon)$ and  $\check{x'}_\epsilon, \check{\b}_\varepsilon, \check{\eta}_\epsilon,\check{0}_\epsilon$ are  defined similarly as $\check{x}_\epsilon$.
%Set
%\begin{equation}\label{macalbarP}
\begin{eqnarray}\label{uthege}
\mathcal{P}_\a(\epsilon)\triangleq\{\eta\in \mathbb{R}^{pm}:  \exists \check{\b}\in \mbox{Im}(\mathcal{B}\bar h)^\epsilon_{\check{x}_\epsilon,\check{\eta}_\epsilon  }   \mbox{ s.t.}&& \mbox{ Im}(\mathcal{B}  h')^\epsilon_{\check{x}_\epsilon,\check{u}_\epsilon, \check{\eta}_\epsilon} \nsubseteq  \mbox{Im}(\mathcal{B}\hat h)^\epsilon_{\check{x}'_\epsilon, \check{u}_\epsilon, \check{\b}}\},\quad
 % &&\mbox{       with  } \check{u}=\{u\},  \,\,\forall u\in \overline {B(0,C)}\}.\quad
\end{eqnarray}
where $u$ is only need to be considered in $\overline{B(0,C_u)  }$.
%\end{equation}

\begin{theorem}\label{iden}
Under Assumptions A1--A2,  control    system \dref{hmodel}
 is identifiable  for any  experiment  $(\chi_0, \{u_t\})\in \mathcal{E}$  if for each $\a\in A_0$,  there exists some $\epsilon>0$ such that    $\underline d^m ( \mathcal{P}_\a(\epsilon )|\mathbb{R}^{pm})> 1/C_w$.
\end{theorem}

We have thus far solved the  identifiability issue. Later,
an implementable algorithm will be provided  in Section \ref{ImA} with the convergence rates explicitly computed.

\subsection{Proofs of Theorems \ref{idensta} and \ref{iden}}

The proofs of Theorems \ref{idensta} and \ref{iden} are similar, so we only give the detailed proof of Theorem \ref{idensta}.

\subsubsection{Theoretical Nonlinear Estimator}\label{nonest}
To design   an  estimator  competent for the identification task,  we need a simple result on functions $f$ and $h$. For this, let
$
\{\Theta_k\subset \mathbb{R}^n, k\geq 0\}
$ be a series of sets  with  $\Theta_0\triangleq  \Theta$ and $\Theta_k\subset \Theta_{k-1}$.
Define % $A_1\triangleq  \left\{(x,x')\in \Theta_0\times \Theta_0:\|x-x'\|\geq c_1\right \}$ and
  $$A_k \triangleq  \left\{(x,x')\in \Theta_{k-1}\times\Theta_{k-1}:\|x-x'\|\geq c_k \right \},\quad k\geq 1,$$
  where $c_k\leq \frac{1}{k}$ is properly small such that $(\{x\} \times \Theta_{k-1}) \cap A_k \neq \emptyset$ for all $x\in \Theta_{k-1}$.
  %Moreover, note that $\mathcal{S}^m$ is compact, by means of the finite covering theorem, the cardinality of $\mathcal{P}_\a$ is sufficient to be taken   finite for each $\a\in A_0$.

\begin{lemma}\label{CoverA}
Let  Assumption A2 hold and $\underline d^m(\mathcal{P}_\a| \mathcal{S} )> 1/C_w,  \forall \a\in A_0$.  Then, for each  $k\geq 1$,  a  finite covering  of $A_k$ in the form
%\begin{equation}\label{coverN}
$\{N_{i,k}\times N'_{i,k}\}_{ n_{k-1}\leq i< n_k}, n_0=1$
%\end{equation}
exists. Moreover,  for every  $ i\in [n_{k-1}, n_k)$, the following two statements hold: \\
(i)   $ N_{i,k}\cap  N_{j,k}=\emptyset, N'_{i,k}\cap  N'_{j,k}=\emptyset$ for $ j\neq i$ and $(N_{i,k}\times N'_{i,k} )\cap A_k \neq \emptyset$;\\
(ii) there exists a finite  set of points  $\Delta^k_i=\prod_{j=1}^m E^{i k}_j  \subset \mathcal{P}_\a$ for some $\a\in A_0$    with
\begin{equation}\label{dEB}
\underline d (E^{i k}_j |\overline {B(0,C_0)})> 1/(C_w-\sigma_k), \sigma_k\in (0,C_w) \quad  \mbox{and} \quad |\Delta^k_i|= \hat{n}_{k},
 \end{equation}
 a sequence
$\{\psi_{is}\in \mathcal{V}^m,      (w^*_{i s,l},  v^*_{i s,l})\in \mathcal{W}\times \mathcal{V}, \mathcal{U}^{i s}_l \subset \mathbb{R}^r \}_{ 1\leq s\leq \hat{n}_{k}, 1\leq l\leq \hat{n}_{is} }$ with  $\{\mathcal{U}^{i s}_l\}$ being some  mutually disjoint sets  that   $\overline {B(0,C_u)}=\sum_{l=1}^{\hat n_{is}} \mathcal{U}^{i s}_l$  and a number  $d_k\in (0,\sqrt{m}\sigma_k)$, where $\hat{n}_{k}$, $\sigma_k$, $d_k $ depend only on $k$ and $\hat{n}_{is} $ depends on $i, s$,   such that for  every %$\eta_{is}\in \Delta^k_i,
$s\in [1,\hat{n}_k]$,  if  $(x,\chi, \psi, w, v,u)\in N_{i,k} \times B( \eta_{is}, d_{k})\times B(\psi_{is},d_{k}) \times B( (w^*_{is,l},v^*_{is,l}), d_{k})\times \mathcal{U}^{i s}_l$, %and $u\in \mathcal{U}^{i s}_l,
$l\in [1, \hat{n}_{is} ]$, then
\begin{eqnarray}\label{uthe}
%\left\{
%\begin{array}{l}
%h_m(x,\chi,\psi)\in B( \beta_{is}, d_k)\\
  h'(x,u, \chi, w,v)\notin (\bigcup\nolimits_{z \in  N'_{i,k}}\mbox{Im}(\hat h_{z, u, \varphi}))\quad \mbox{with}\quad \varphi=\bar h(x,\chi,\psi).
%\end{array}.
%\right.
\end{eqnarray}
\end{lemma}

%\subset S^m\times h_m(\Theta,S^m)

\begin{proof}
Note that  $\underline d^m (\mathcal{P}_\a|\mathcal{S})> 1/C_w$
for each $\a\in A_0$. Take a $\sigma_\a \in (0,C_w)$  and a sequence  $\{\bar{E}^{\a }_j\}_{1\leq j\leq m}$   such that $\prod_{j=1}^m \bar{E}^{\a }_j \subset \mathcal{P}_\a$ and
$\min_{j\in [1,m] }\underline d (\bar{E}^{\a }_j| \overline {B(0,C_0)})> 1/(C_w-2\sigma_\a)$. Then, for every  $j\in [1,m]$ and  $z\in \overline {B(0,C_0)}$, $z\in  B(e^{\a }_j, C_w-2\sigma_\a)
$ for some $e^{\a }_j \in \bar{E}^{\a }_j$.
Consequently, by the compactness of $\overline {B(0,C_0)}$,  there exists a finite set  $E^{\a }_j=\{e^{\a }_j\in \bar{E}^{\a }_j\}$ such that
$$
\overline {B(0,C_0)}\subset \bigcup\nolimits_{ e^{\a }_j\in E^{\a }_j} B(e^{\a }_j,C_w-2\sigma_\a).
$$
 So, for every $ \a\in A_0$ and    $j\in [1, m]$,
 \begin{equation}\label{dEa}
 \underline d (E^{\a }_j| \overline {B(0,C_0)})> 1/(C_w-\sigma_\a).
  \end{equation}
  Clearly, $\Delta_\a \triangleq\prod_{j=1}^m E^{\a }_j \subset \mathcal{P}_\a$ and $\hat{n}_{\a}\triangleq |\Delta_\a|$ is finite as well.

Now, fix $k\geq 1$.    Given
 $\alpha=(x,x')\in A_k$,
 \dref{macalP} shows that  for any $\eta_{\a s}\in \Delta_\a$ and $u\in \overline {B(0,C_u)}$,  there are some  $\psi_{\a s}\in \mathcal{V}^m$ and $( w^*_{\a s,u},v^*_{\a s,u})\in  \mathcal{W}\times  \mathcal{V}$ such that
\begin{eqnarray*}
\left\{
\begin{array}{l}
\b_{\a s}=\bar h(x, \eta_{\a s},\psi_{\a s}) \\
h'(x, u, \eta_{\a s},w^*_{\a s,u},v^*_{\a s,u} )\notin  \mbox{Im}(\hat h_{x', u, \b_{\a s}})
\end{array}.
\right.
\end{eqnarray*}
In view of  Assumption A2, % $f, h$ are continuous and $h^{-1}$ is % bounded-valued and upper semicontinuous, so
 $h', \bar h$ are continuous and $\hat h$ is upper semicontinuous and bounded-valued.
%.  Since $\mathcal{V}$ is compact, $\bar h^{-1}(z,\varphi,\mathcal{V}^m)$ is bounded at every $(z,\varphi)$.
%$f, h, \xi$ are continuous, so $\mbox{Im}(\bar h_{z, \varphi})$ is bounded as well at every $(z,\varphi)$ due to
By the  compactness of $\mathcal{W}$ and $ \mathcal{V}$, %for any fixed $ u\in \mathbb{R}^r$,
 $\mbox{Im}(\hat h_{x, u, \varphi})$ is upper semicontinuous and bounded-valued  at $(x, u, \varphi)$ as well. Because $\overline {B(0,C_u)}$ is compact, for each  $s\in [1,\hat{n}_{\a}]$, there exist  some mutually disjoint sets  $\{\mathcal{U}^{\a s}_l\}_{1\leq l\leq \hat n_{\a s}}$ with $\overline {B(0,C_u)}=\sum_{l=1}^{\hat n_{\a s}} \mathcal{U}^{\a s}_l$,  some points $\{( w^*_{\a s,l},v^*_{\a s,l})\in  \mathcal{W}\times  \mathcal{V}\}_{1\leq l\leq \hat n_{\a s}}$, some neighbourhoods $B^\a_{x}$ and $B^\a_{x'}$
of  $x$ and $x'$ respectively and a number $\varepsilon_{\a }>0$,
  such that
  if $u\in  \mathcal{U}^{\a s}_l$,   then
%\begin{eqnarray*}
%h'(x, u, \eta_{\a s},w^*_{\a s,l},v^*_{\a s,l} )\notin  \mbox{Im}(\hat h_{x', u, \b_{\a s}}).
%\end{eqnarray*}
 %Note that $|\Delta_\a|$ is finite,
% Therefore, there exist some neighbourhoods $B^\a_{x}$ and $B^\a_{x'}$
%of  $x$ and $x'$ respectively and a number $\varepsilon_{\a }>0$ such that for any $s\in [1,\hat{n}_{\a}]$ and $u\in  \mathcal{U}^{\a s}_l$, $1\leq l\leq \hat n_{\a s},$
%\begin{equation*}
%\left(\bigcup_{z\in B^\a_{x,k}} \bar h'(x, \b_{\a s}, \eta_{\a s},w^*_{\a s},v^*_{\a s} )\right)\cap \left( \bigcup_{z\in B^\a_{x',k}}\mbox{Im}(\bar h_{z,  \b_{\a s}})\right)=\emptyset.
%\end{equation*}
%By the %upper semicontinuity
%continuity of $\bar h'$   in variables $\eta, w,v$ and the continuity of $\xi$ at $\b_\a$, there is a $d_{\a k}\in (0,\epsilon)$ satisfying
\begin{equation}\label{expoint}
h'(B^\a_{x},u, \bar D_{\a s,l})\cap ( \bigcup\nolimits_{ z\in B^\a_{x'},\varphi \in B( \beta_{\a s}, \varepsilon_{\a })  }\mbox{Im}(\hat h_{z, u, \varphi}))=\emptyset,\quad  1\leq l\leq \hat n_{\a s},
\end{equation}
where $\bar D_{\a s,l}\triangleq  B( \eta_{\a s}, \varepsilon_{\a })\times B( (w^*_{\a s,l},v^*_{\a s,l}), \varepsilon_{\a })$. Note that $B^\a_{x}, B^\a_{x'}$ and $\varepsilon_{\a }$ are taken independent of $s\in [1,\hat{n}_{\a}]$.
In addition, as long as $B^\a_{x}$ is sufficiently small, there is  a  $d_{\a }\in(0,\varepsilon_{\a })$ satisfying
\begin{equation}\label{expoint1}
%\bigcup_{(z,\chi, \psi)\in B^\a_{x,k}\times B( \eta_{\a s}, d'_{\a k})\times B( 0, d'_{\a k})  }
 \bar h(B^\a_{x},B( \eta_{\a s}, d_{\a }), B(\psi_{\a s}, d_{\a }) \subset B( \beta_{\a s}, \varepsilon_{\a }),\quad \forall s\in [1,\hat{n}_{\a}].
\end{equation}

Now, for any $(x,x')\in A_k$,   we  find an open set $B^\a_{x}\times B^\a_{x'}$ fulfilling
\dref{expoint} and \dref{expoint1} for all $\eta_{\a s}\in \Delta_\a, s\in [1,\hat{n}_{\a}]$.
Therefore, the compact set $A_k$ can be covered by some  finite open sets $\{B_{i,k}\times B'_{i,k}, \bar n_{k-1}\leq i< \bar n_k\}$ ($\bar n_0=1$), where for each $i\in [\bar n_{k-1},\bar n_k)$, it corresponds to a set $\Delta^k_i=\prod_{j=1}^m E^{\a}_j =\Delta_\a$ for some $\a\in A_k$,  % a series of sets $\{\mathcal{U}^{i s}_l\}_{ 1\leq s\leq \hat{n}_{k},1\leq l\leq \hat n_{i s}}$ refining $\overline {B(0,C)}\subset \mathbb{R}^r$,
 a sequence   $\{( \psi_{i s}, w^*_{i s,l},v^*_{i s,l})\in  \mathcal{V}^m\times \mathcal{W}\times  \mathcal{V}, \mathcal{U}^{i s}_l\subset \mathbb{R}^r\}_{  1\leq s\leq \hat{n}_{k},1\leq l\leq \hat n_{i s}}$ with $\overline {B(0,C_u)}=\sum_{l=1}^{\hat n_{is}} \mathcal{U}^{i s}_l$    and some numbers $\sigma_k, d_k, \varepsilon_{k}$ with  $0<d_{k}<\min\{\varepsilon_{k},  \sqrt{m}\sigma_k\}$, such that for any $(x,\varphi)\in B_{i,k}\times B( \b_{is}, \varepsilon_{k}) $ with
 $\b_{is}=\bar h(x,\eta_{is},\psi_{i s})$ and $\eta_{is}\in \Delta^k_i, s\in [1,\hat{n}_{k}]$,
 \begin{eqnarray}\label{Bep}
\bar  h(B_{i,k}, B(\eta_{is},d_{k}), B(\psi_{i s},d_{k})) \subset B( \beta_{is}, \varepsilon_{k})
\end{eqnarray}
 and when $u\in  \mathcal{U}^{i s}_l$, $1\leq l\leq \hat n_{i s},$
\begin{eqnarray}\label{Bimh}
%\left\{
%\begin{array}{l}
%\bigcup_{(\psi,\chi)\in B(\eta_{js},d'_{k})\times B(0,d'_{k})}
%\bar  h(B_{j,k}, B(\eta_{js},d_{k}), B(\psi_{j s},d_{k})) \subset B( \beta_{js}, \varepsilon_{k})\\
\left( h'(x,u,  B( \eta_{is}, d_{k}),B( (w^*_{is,l},v^*_{is,l}), d_{k}))\right)\cap (   \bigcup\nolimits_{z\in  B'_{i,k} }\mbox{Im}(\hat h_{z, u, \varphi}))=\emptyset.
%\end{array}.
%\right.
\end{eqnarray}
So, if for some $s\in [1,\hat{n}_{k}], l\in [1, \hat n_{i s}]$,    $\varphi=\bar h(x,\chi,\psi)$, $u\in  \mathcal{U}^{i s}_l$ and
  $$(x,\chi, \psi, w, v)\in B_{i,k} \times B( \eta_{is}, d_{k})\times B(\psi_{i s},d_{k}) \times B( (w^*_{is,l},v^*_{is,l}), d_{k}),$$
 then  by  \dref{Bep}--\dref{Bimh},
%\begin{eqnarray*}
 $ h'(x,u, \chi, w,v)\notin (\bigcup_{z \in   B'_{i,k}}\mbox{Im}(\bar h_{z, u,\varphi}))$.
%\end{eqnarray*}

Finally,
 let $\{ N_{i,k}\}$ and $\{ N'_{i,k}\}$ be a series of  refined sets of $\{B_{j,k}\}$ and $\{B'_{j,k}\}$, respectively, such that $ N_{i,k}\cap  N_{i',k}=\emptyset$ and $N'_{i,k}\cap  N'_{i',k}=\emptyset, i'\neq i$. Clearly, $\{ N_{i,k}\times  N'_{i,k}\}_{ i\in [n_{k-1},n_k)}$ is a finite covering  of $A_k$. Without loss of generality, let $(N_i\times N'_i )\cap A_k \neq \emptyset, i\in [n_{k-1},n_k)$. Since every $ N_{i,k}\times  N'_{i,k}\subset B_{j,k}\times  B'_{j,k}$ for some $j\in [\bar n_{k-1},\bar n_k)$, \dref{uthe} follows  immediately. Besides, \dref{dEB} holds by \dref{dEa}.
\end{proof}

We now provide a  theoretical estimator  to identify parameter $\theta$.
Rewrite the finite covering  of $A_k, k\geq 1$ in Lemma \ref{CoverA} by
 \begin{equation}\label{group}
 \{\bar N_{i,k}\times\bar N'_{ij,k}, 1\leq j\leq m_{k,i}\},\quad  m_{k-1} \leq i <  m_k\,\,(m_0=1).
\end{equation}
So, $\bar N_{i,k}\cap \bar N_{j,k}=\emptyset, \forall  i\neq j$ and $\sum_{i=m_{k-1}}^{m_k-1} m_{k,i}=n_k-n_{k-1}$. Let $\theta_0$ be the center of $\Theta$.

 %$\{\bar N^{x}_{j}\times \bar N^{x'}_{j}, \bar n_{k-1} \leq j <  \bar n_k\}$
%is a refined finite covering     of $A_k$ determined by $\{N_i, n_{k-1}\leq i < n_k\}$ such that $\bar N^{x}_{j}\times \bar N^{x'}_{j} \cap  N^{x}_{i}\times N^{x'}_{i}$ is either $\emptyset$ or $\bar N^{x}_{j}\times \bar N^{x'}_{j} $ for any $i$ and $j$. Divide $\{\bar N^{x}_{j}\times \bar N^{x'}_{j}\}$  into several groups  by
%$$
 %\{\bar N^{x}_{i}\times \bar N^{x'}_{i,j}, 1\leq j\leq n_{k,i}\},\quad \bar n_{k-1} \leq i <  \bar n_k,
%$$
%where
%  \begin{eqnarray*}
%\left\{
%\begin{array}{ll}
%\bigcup_{i\in [\bar n_{k-1}, \bar n_k)}\bigcup_{ j\in [1, n_{k,i}]} \bar N^{x'}_{i,j}=\bigcup_{i\in [\bar n_{k-1}, \bar n_k)} N^{x'}_{i}\\
%\bigcup_{i\in [\bar n_{k-1}, \bar n_k)}\bigcup_{ j\in [1, n_{k,i}]} \bar N^{x}_{i}\times \bar N^{x'}_{i,j}=\bigcup_{i\in [\bar n_{k-1}, \bar n_k)} N^{x}_{i}\times N^{x'}_{i}
%\end{array}.
%\right.
%\end{eqnarray*}

% $ \bigcup_{i\in [\bar n_{k-1}, \bar n_k)}\bigcup_{ j\in [1, n_{k,i}]} \bar N^{x'}_{i,j}=\bigcup_{i\in [\bar n_{k-1}, \bar n_k)} N^{x'}_{i}$ and  $ \bigcup_{i\in [\bar n_{k-1}, \bar n_k)}\bigcup_{ j\in [1, n_{k,i}]} \bar N^{x}_{i}\times \bar N^{x'}_{i,j}=\bigcup_{i\in [\bar n_{k-1}, \bar n_k)} N^{x}_{i}\times N^{x'}_{i}$.

\noindent\textbf{Algorithm:}

\begin{description}
\item[Step 1] Let $t_0=0$, $\hat\theta_0=\theta_0$ and  $\hat\Theta_{i,0}= \bar N_{i,1}\times \Theta_0$ for all $i=1, \ldots, m_1-1$.

\item[Step 2] For   $t> t_{k-1}, k\geq 1$, if $\hat\Theta_{i,{t-1}} \neq (\hat\Theta_{i,{t_{k-1}}}  \setminus  A_k)$  for all $i\in [m_{k-1},  m_k)$, denote
\begin{equation*}
J^{k}_{it}\triangleq \{j\in [1, m_{k,i}] : y_t \notin \bigcup\nolimits_{    z\in \bar N'_{ij,k}}\mbox{Im}(\bar h_{z, u_{t-1}, \varphi_{t-1}})\},\quad m_{k-1} \leq i < m_k.
\end{equation*}
Let
  $\hat\Theta_{i,t}=\hat\Theta_{i,t-1}\setminus ( ( \bar N_{i,k}\times  \bigcup_{j\in J^k_{it} } \bar N'_{ij,k})\cap A_k)$,  where $ i\in [m_{k-1},  m_k).$
  If   for all  $i\in [m_{k-1},  m_k)$,  $\hat\Theta_{i,t}\neq (\hat\Theta_{i,{t_{k-1}}}  \setminus  A_k)$,  set
\begin{equation}\label{varpgih}
\hat \theta_t=\hat \theta_{t-1}.
\end{equation}

\item[Step 3] For   $t>t_{k-1}, k\geq 1$, if  $\hat\Theta_{i,t} =(\hat\Theta_{i,{t_{k-1}}}  \setminus  A_k)$ for some  $i\in [m_{k-1},m_k)$,  take a point $(x,x)\in \hat\Theta_{i,t}$ and set
\begin{equation}\label{estmate}
\hat\theta_t=x.
\end{equation}

Set $\Theta_k=\overline{B(\hat\theta_t,c_k)}$,     $\hat\Theta_{i,t}=\bar N_{i,k+1}\times \Theta_k$ for $i=m_k, \ldots, m_{k+1}-1$,  and $t_{k}=t$.

\end{description}

\begin{remark}\label{welldefine}
If $t_{k-1}<\infty$ for some $ k\geq 1$, then the algorithm implies that $\Theta_{k-1}$, $A_k$ and $\{\bar N_{i,k}\times \bar N_{ij,k}\}_{ i\in [1,m_k),j\in[1,m_{k,i}]}$ are well defined. As a result,   in Lemma \ref{CoverA},  $\{(\eta_{is}, \psi_{i s}, w^*_{is,l},v^*_{is,l}), \mathcal{U}^{i s}_l\}_{i\in [1,n_k), s\in [1,\hat{n}_{k}], l\in [1,\hat{n}_{is} ]}$ are also  well defined.
\end{remark}

For each $k\geq 1$ and $i\in [n_{k-1},n_k)$, denote $\Gamma^k_{i}\triangleq \bigcup_{s=1}^{\hat{n}_{k}} \bigcup_{l=1}^{\hat{n}_{is}}  D^k_{is,l}$ with
\begin{equation}\label{Dkis}
 D^k_{is,l}\triangleq B(\eta_{is}, d_{k})\times  B(\psi_{i s}, d_{k})\times B( (w^*_{is,l},v^*_{is,l}), d_{k})\times \mathcal{U}^{i s}_l,
\end{equation}
where  $\eta_{is}, \psi_{is}, w^*_{is,l},v^*_{is,l}, \mathcal{U}^{i s}_l$  and  $d_{k}>0$ are defined in Lemma \ref{CoverA}.

\begin{lemma}\label{conthe}
Let  $(\varphi_0, \{u_t\}) \in \mathcal{E}$ be an experiment designed   that for each $k\geq 1$, if $t_{k-1}<\infty$ almost surely and $t_k=\infty$ on a set $D$ with $P(D)>0$, then
\begin{equation}\label{Tki}
T^k_i\triangleq\{t>t_{k-1}: (  \chi_{t-1}, \psi_{t-1},  w_{t}, v_{t}, u_{t-1} )\in \Gamma^k_{i}\}\neq \emptyset
\end{equation}
will hold almost surely on $D$  for all  $i\in [n_{k-1},n_k)$ satisfying $\theta\in N_{i,k}$, where $\psi_{t-1}\triangleq (v_{t-1},\ldots,v_{t-m})^T$. % almost surely on $D$.
%:\\
%(i)  for each  $i\in [\bar n_{k-1},\bar n_k)$, $T^\eta_i\triangleq\{t: \varphi_t\in B(\b_i, d_k)\}\neq \emptyset $ almost surely;\\
%(ii) for each $i\in [\bar n_{k-1},\bar n_k)$, $T^\varsigma_i\triangleq\{t: (u_t, \varphi_t)\in B(\varsigma_i, d_k)\times B(\b_i, d_k)\}\neq \emptyset $ almost surely.\\
Then, under the  conditions  of
Lemma \ref{CoverA}, the nonlinear estimator constructed by \dref{varpgih}--\dref{estmate} satisfies
$
\lim_{t\rightarrow \infty}\hat\theta_t=\theta
%on}\quad D_\eta,
$ almost surely.
%where $D_\eta\triangleq \left\{\omega: \lim_{t\rightarrow \infty} \eta_t=\infty \right\}$ with $\eta_t\triangleq \sum_{i=m}^{t-1} I_{\Omega_i}$.
\end{lemma}

\begin{proof}
We first show  that under an experiment $(\varphi_0, \{u_t\}) \in \mathcal{E}$ designed in this lemma, the nonlinear algorithm will fulfill  $t_k<\infty$ and $\theta\in \Theta_k$  for all $k\geq 0$ almost surely (this also means $\Theta_k$ are well defined for all $k$ almost surely). Since  $t_0=0$ and $\Theta_0=\Theta$,
suppose for some $k\geq 1$, $t_i<\infty$  and $\theta\in\Theta_i$ for all $i\in [0, k-1]$ almost surely. We claim that $t_k<\infty$ a.s. for this $k$. Otherwise, there is a set $D$ with $P(D)>0$ such that  $t_k=\infty$ on $D$. Now, $t_{k-1}<\infty$, by Remark \ref{welldefine}, $\Theta_{k-1}$  and $A_k$ are well defined.
 Note that $(\{\theta\} \times \Theta_{k-1}) \cap A_k \neq \emptyset$ and hence
$\theta\in \bar N_{\varsigma,k}$ for some $\varsigma\in [m_{k-1},m_k)$. Let
$$
I_k\triangleq\{i\in[n_{k-1},n_k): N_{i,k}\times N'_{i,k}=\bar N_{\varsigma,k}\times \bar N'_{\varsigma j,k}, 1\leq j\leq m_{k,\varsigma}\}.
$$
So, $|I_k|=m_{k,\varsigma}>0$.
The experiment ensures
$T^k_i\neq \emptyset$ for all $i\in I_k$ on $D$ almost surely. Consequently, for each $j\in [1, m_{k,\varsigma}]$ which corresponds to an integer $i(j)\in I_k$,  there exist some random integers $t(j), s(j), l(j)$ taking values in  $  T^k_{i(j)}, [1, \hat{n}_{k}]$ and $[1,\hat{n}_{i(j) s(j)}]$  respectively  such that
 $$(  \chi_{t(j)-1}, \psi_{t(j)-1},  w_{t(j)}, v_{t(j)},  u_{t(j)-1} )\in D^k_{i(j)s(j),l(j)}\quad \mbox{a.s on } D.$$ Considering $\theta\in \bar N_{\varsigma,k}=N_{i(j),k}$,
by statement (ii) of Lemma \ref{CoverA},% with probability $P(D)$ on $D$,
\begin{eqnarray}\label{yNsj}
%\left\{
%\begin{array}{l}
%\varphi_{t(j)-1}= (h_m( \theta,\chi_{t(j)-1},\psi_{t(j)-1} )%\in B( \beta_{i(j)s(j)}, d_k)
%\\
y_{t(j)}= h'(\theta, u_{t(j)-1},\chi_{t(j)-1},w_{t(j)},v_{t(j)})\notin\bigcup\nolimits_{x\in \bar N'_{\varsigma j,k}}\mbox{Im}(\hat h_{x, u_{t(j)-1},\varphi_{t(j)-1}})
%\end{array},\,\,\mbox{a.s.}.
%\right.
\end{eqnarray}
holds almost surely on $D$,  where $\varphi_{t(j)-1}= \bar h( \theta,\chi_{t(j)-1},\psi_{t(j)-1} )$.

Now, by Step 3 of   the algorithm, it is clear that for each $i\in [m_{k-1},m_k-1]$,
\begin{equation}\label{Is}
\emptyset \neq \{(x,x')\in\bar N_{i,k}\times \Theta_{k-1}:x=x'\}\subset (\hat\Theta_{i,{t_{k-1}}}  \setminus  A_k).
\end{equation}
%and by denoting $\Theta_{0,s}\triangleq N_s\times \Theta$,
%$$
%\Theta_{k,s}\subset\Theta_{k-1,s}=\hat\Theta_{s,t_{k-1}-1}.$$
Since $t_k=\infty$ on $D$,  $\hat\Theta_{\varsigma,{t}} \neq (\hat\Theta_{\varsigma,{t_{k-1}}}  \setminus  A_k)%, m_{k-1} \leq i < m_k
$ for all $t\geq t_{k-1}$ on $D$.
%However,   % there  is a random sequence  $\{t(j)> t_{k-1}\}_{ j\in [1,m_{k,\varsigma}]}$ such that
%$$
%y_{t(j)}\notin\bigcup\nolimits_{x\in \bar N'_{\varsigma j,k}}\mbox{Im}(\hat h_{x,u_{t(j)-1}, \varphi_{t(j)-1}}),\quad \mbox{a.s. on }D.
%$$
%satisfying \dref{uthe}. Note that $\theta\in N_s$,
%So,
    Denote $\bar t_{k-1}\triangleq\max_{1\leq j\leq m_{k,\varsigma}}t(j)$, then  \dref{yNsj} yields $J^k_{\varsigma \bar t_{k-1}}=\{1,\ldots, m_{k,\varsigma}\}$ a.s. on D.   So, by Step 2,  %with probability $1$ on $D$,
$$
\hat\Theta_{\varsigma,\bar t_{k-1}}=\hat\Theta_{\varsigma,{t_{k-1}}} \setminus( (\bar N_{\varsigma,k}\times \bigcup\nolimits_{1\leq j\leq m_{k,\varsigma}}\bar N'_{\varsigma j,k})\cap A_k)= (\hat\Theta_{\varsigma,{t_{k-1}}} \setminus A_k),
$$
on $D$ almost surely,  which leads to a contradiction. Therefore, $t_k<\infty$ almost surely. Moreover, Step 3 implies that $\Theta_k$ is well defined almost surely.

The remainder is devoted to verifying $\theta\in \Theta_k$ on $\{t_k<\infty\}$. Take a trajectory on which $t_k<\infty$. The follow-up arguments are restricted on this trajectory.   Denote
$$
I'_k\triangleq\{i\in [m_{k-1},m_k): \mbox{diam}(\theta,\bar N_{i,k})> c_k\},
$$
which means for each $i\in I'_k$,
there is a point $x\in \bar N_{i,k}$ such that $d(x,\theta)> c_k$.
Recall that $\theta\in \Theta_{k-1}$,  then     $(x,\theta)\in A_k$ and thus
$\theta\in \bar N'_{ij,k}$ for some $j\in [1,m_{k,i}]$ due to
$$A_k\subset \bigcup\nolimits_{i\in [m_{k-1},m_k), j\in [1,m_{k,i}]} \bar N_{i,k} \times \bar N'_{ij,k}.$$
So, for all $t>t_{k-1}$,
 $y_t\in\bigcup_{z\in \bar N'_{ij,k}}\mbox{Im}(\hat h_{z, u_{t-1}, \varphi_{t-1}})$,
which implies $((\bar N_{i,k}\times \bar N'_{ij,k})\cap A_k)\neq\emptyset$ belongs to $ \hat\Theta_{i,t}$. Consequently, $\hat\Theta_{i,t}\neq(\hat\Theta_{i,{t_{k-1}}} \setminus A_k)$ for all $t> t_{k-1}$ whenever $i\in I'_k$. Now, $t_k<\infty$, so any index $\varsigma$ causes $\hat\Theta_{\varsigma,t_k}=(\hat\Theta_{\varsigma,{t_{k-1}}} \setminus A_k)$ at Step 3 must satisfy $\varsigma\in [m_{k-1},m_k)\backslash I'_k$.  Hence,  diam$(\theta,\bar N_{\varsigma,k})\leq c_k.$
%there exists some $l\in [m_{k-1},m_k)\backslash I'_k$ that $\hat\Theta_{l,t_k}=(\hat\Theta_{l,{t_{k-1}}} \setminus A_k)$ and hence  diam$(\theta,\bar N_{l,k})\leq c_k.$
Moreover, because of \dref{Is}, $\hat\theta_{t_k}$  in   \dref{estmate} is well defined at Step 3 and $\hat\theta_{t_k}\in \bar N_{\varsigma,k}$. As a result,
\begin{equation}\label{errork}
\|\theta-\hat\theta_{t_k}\|\leq c_k\leq \frac{1}{k},
\end{equation}
which immediately yields that $\theta\in \Theta_k=\overline{B(\hat\theta_{t_k},c_k)}$ on the fixed trajectory.

Therefore, we have verified that $t_k<\infty$ and $\theta\in \Theta_k$  for all $k\geq 0$ almost surely and  hence \dref{errork} holds for all $k\geq 1$ accordingly.
Since Step 2 in the algorithm implies that for each $k\geq 1$,
\begin{eqnarray*}
\hat\theta_t=\hat\theta_{t_{k-1}},\quad t_{k-1}\leq t\leq t_k-1,
\end{eqnarray*}
the lemma is thus proved by letting $k\rightarrow\infty$.
\end{proof}

\subsubsection{Proofs of the   Theorems}
%For  $\|\chi_t\|\leq C$, it is easy to see that  %then $\|f(\theta,u_t,\chi_t) \|\leq C_f$ and
%$
%\|f(\theta,u_{t-j},\chi_{t-j})\|=\|x_{t-j+1}-w_{t-j+1}\|\leq C+C_w,\quad j=1,\ldots,m.
%$$
%So, when the carried  experiment $(\chi_0, \{u_t\})\in \mathcal{E}$  is $C$-recurrent,
%$$
%P(\|f(\theta,u_{t-j},\chi_{t-j})\|\leq C+C_w, j=1,\ldots,m,\mbox{   i.o.})=1.
%$$
%According to \dref{eCe},  for every $\a\in A_0$,
%\begin{equation}\label{einfB}
%P\left(\bigcup_{l\in[1,n_C]} \{e^\a_{jl}\}\subset B(f(\theta,u_{t-j},\chi_{t-j}),C_w-\sigma), j\in [1,m], \mbox{   i.o.} \right)=1.
%\end{equation}

%Some notations are needed in the sequel. For each $t\geq0$, define  $\Omega_t\triangleq \{\|\chi_{t}\|\leq C\}$ and  $ f_{t}\triangleq f(\theta, u_{t}, \chi_{t})$. Let
Some notations are needed in the sequel. For each $t\geq0$, denote  $ f_{t}\triangleq f(\theta, u_{t}, \chi_{t})$ and  $\Omega_t\triangleq \{\|\chi_t\|\leq C\}$. Let
 \begin{eqnarray}\label{Ft}
\mathcal{F}_t\triangleq\left\{
\begin{array}{ll}
\sigma\{\chi_0, u_0, u_{i}, w_{i}, v_i,  i\in [1,t]\},&t\geq 1\\
%\{\Omega, \emptyset\},&t=0
\sigma\{\chi_0, u_0\},&t=0
\end{array}.
\right.
\end{eqnarray}
Write $E^{i k}_j =\{e^{ik}_{s,j}\}_{1\leq s\leq |E^{i k}_j |}, j\in [1,m]$ in Lemma \ref{CoverA}. Clearly, $\prod_{j=1}^m|E^{i k}_j |=\hat n_k$.
%Write $\eta^k_{is}=(e^{ik}_{s,1},\ldots,e^{ik}_{s,m})^T\in\Delta^k_i$, $s\in [1,\hat{n}_k], i\in [n_{k-1},n_k)$, where $\Delta^k_i$ is defined in Lemma \ref{CoverA}.
In addition, by  Assumption A2, %$\chi_t\|\leq C$ yields
  \begin{equation}\label{f<}
f_i=\|f(\theta, u_{i}, \varphi_{i})\|\leq  C_0,\quad i=t-1,\ldots,t-m, \quad \mbox{on}\quad \Omega_{t-m}.
   \end{equation}

\begin{lemma}\label{stiFt}
Let $t_{k-1}<\infty, k\geq 1$ and $ i\in [n_{k-1},n_k)$. % let $\Delta^k_i, i\in [n_{k-1},n_k)$ be defined in Lemma \ref{CoverA}.
If $\underline d^m(\mathcal{P}_\a| \mathcal{S} )> 1/C_w$ for all $ \a\in A_0$ and  Assumption A2 holds,  then for each $t\geq m$, %on set $\Omega^{ik}_{t,m-j}$,
  there are some random integers $\{s_{t_j}\in \mathcal{F}_{t-j}\}_{j\in [1,m]}$ taking values in $\mathcal{N}_{k,j}=\{1,\ldots, |E^{i k}_j |\}$  on $\Omega_{t-m}$ such  that % $  e^{ik}_{s_{t},j} \in \mathcal{F}_{t+m-j}$. Moreover,
   % $ I_{\Omega_t}\leq    I_{\{f_{t}\in B^{ik}_{m,t}\}}$ and for all $j\in [1,m-1]$,
%\begin{eqnarray}\label{ftjC}
%\{x_{t-j} \in B(e^{ik}_{s_t,j},d'_k) \}\subset \{  f(\theta, u_{t-j}, \chi_{t-j})\in B(e^{ik}_{s_{t-m},j},  C_w-\sigma/2)\},\quad j\in [1,m],
%\end{eqnarray}
\begin{eqnarray}\label{ftjC}
%\left\{
%\begin{array}{l}
%I_{\Omega^{ik}_{t,j}}
I_{\Omega_{t-m}}\leq I_{ \{  f_{t-j}\in B^{ik}_{j,t}   \}},\quad j\in [1,m],
%I_{\Omega_t}\leq    I_{\{f_{t}\in B^{ik}_{m,t}\}}
%\end{array},
%\right.
\end{eqnarray}
  where % $\Omega^{ik}_{t,j}\triangleq {\{x_{t+l} \in B(e^{ik}_{s_{t_{m-l+1}},m-l+1}, \bar d_{k}), l\in [1,j] \}}\cap {\Omega_t}$\footnote{$\Omega^{ik}_{t,0}=\Omega_t$} for $j\in [0,m-1]$,
 $B^{ik}_{j,t}\triangleq B(e^{ik}_{s_{t_j},j},  C_w-\sigma_k), j\in [1,m]$ and $\sigma_k$ is defined in Lemma \ref{CoverA}.% and $\bar d_{k}\in (0,\sigma_k)$.
\end{lemma}

\begin{proof}
Since $t_{k-1}<\infty,k\geq 1$, by the algorithm and Lemma \ref{CoverA},   all the quantities appearing in the lemma are well defined. Fix
 $i\in [n_{k-1},n_k)$. Note that  by \dref{f<},  $(f_{t-1},\ldots,f_{t-m})\in\mathcal{S}$ on $\Omega_{t-m}$, then
for $j=1,\ldots,m$, define
 \begin{eqnarray}\label{stj}
s_{t_j} \triangleq
\left\{
\begin{array}{ll}
\min\{s\in \mathcal{N}_{k,j}:f_{t-j}\in B(e^{ik}_{s,j},  C_w-\sigma_k) \},& \mbox{on }\Omega_{t-m}\\
0,&\mbox{on }\Omega^c_{t-m}
\end{array}.
\right.
\end{eqnarray}
Random sequence $\{s_{t_j}\}_{ j\in [1,m]}$ is well defined  on $\Omega_{t-m}$ since  $\underline d^m (\Delta^k_i|\mathcal{S})>1/(C_w-\sigma_k)$ by Lemma \ref{CoverA}. So,  $s_{t_j}\in \mathcal{F}_{t-j}$ and  \dref{ftjC} follows immediately.
\end{proof}

\begin{lemma}\label{inputsta}
Let $(\chi_0, \{u_t\})\in \mathcal{E}$ be a $C$-recurrent experiment and $\underline d^m(\mathcal{P}_\a| \mathcal{S} )> 1/C_w$ for all $ \a\in A_0$. %and  $\Delta^k_i, \sigma_k, d_k$ be defined in Lemma \ref{CoverA}.
Then, under    Assumptions A1--A2, for each $k\geq 1$,
\dref{Tki}
holds for all  $i\in [n_{k-1},n_k)$ a.s. whenever  % $d_k<\sqrt{m}\sigma_k$ and
$t_{k-1}<\infty$ a.s..
\end{lemma}

\begin{proof}
Fix a $k\geq 1$ that $t_{k-1}<\infty$ a.s. and take an integer $ i\in [n_{k-1},n_k)$.
%Without loss of generality, assume    %$u\in \mathcal{U}^{i s}_l,
 % all   $(\eta_{is},\psi_{is}, w^*_{is,l},v^*_{is,l}), s\in [1,\hat{n}_{k}], l\in [1,\hat{n}_{is}]$ are   distinct and           $d_{k}$ is sufficiently small, so that
Let
\begin{eqnarray}
\left\{
\begin{array}{ll}
D^k_{i1}(1)\triangleq   B(\eta_{i1}, d_{k})\times  B(\psi_{i1}, d_{k})\\
D^k_{is}(1)\triangleq  ( B(\eta_{is}, d_{k}) \backslash \bigcup\nolimits_{j=1}^{s-1}    B(\eta_{ij}, d_{k})  )\times  B(\psi_{is}, d_{k}), & s\in [2,\hat n_k]
\end{array}
\right.
\end{eqnarray}
and $D^k_{is,l}(2)\triangleq  B( (w^*_{is,l},v^*_{is,l}), d_{k})\times \mathcal{U}^{i s}_l$. Recall that $\{\mathcal{U}^{i s}_l  \}$ are mutually disjoint,
%$  D^k_{is}(1)\triangleq   B(\eta_{is}, d_{k})\times  B(\psi_{is}, d_{k}),  $   and
%$$
%\bar D^k_{i1}(1)\triangleq  D^k_{i1}(1)\quad\mbox{and}\quad   \bar D^k_{is}(1)\triangleq  D^k_{is}(1)\backslash \bigcup\nolimits_{j=1}^{s-1} %D^k_{ij}(1), s\in [2,\hat n_k].
%$$
then
%\begin{eqnarray}\label{Bis}
$ \{ D^k_{is}(1) \times  D^k_{is,l}(2)\}_{ s\in [1,\hat{n}_{k}], l\in [1,\hat{n}_{is}]}$
%\end{eqnarray}
%defined by \dref{Dkis}
are mutually disjoint as well.
As a result, for  $t\geq m$, %compute %by Assumption A1,
\begin{eqnarray*}\label{PGamma}
&&P\left((  \chi_{t}, \psi_{t},  w_{t+1}, v_{t+1},u_t  )\in \Gamma^{k}_{i}|\mathcal{F}_{t-m} \right)I_{\Omega_{t-m}}\nonumber\\
&=&\sum_{s=1}^{\hat{n}_{k}} \sum_{l=1}^{\hat{n}_{is}} P\left((  \chi_{t}, \psi_{t},  w_{t+1}, v_{t+1},u_t  )\in D^k_{is}(1) \times  D^k_{is,l}(2)|\mathcal{F}_{t-m}\right)I_{\Omega_{t-m}}\nonumber\\
&=&\sum_{s=1}^{\hat{n}_{k}} \sum_{l=1}^{\hat{n}_{is}}   E(   I_{\{\chi_{t}, \psi_{t})   \in D^k_{is}(1) \}}
 P ((w_{t+1}, v_{t+1},u_t )\in  D^k_{is,l}(2)|   \mathcal{F}_{t} )  |\mathcal{F}_{t-m}       )I_{\Omega_{t-m}},\,\, \mbox{a.s.}.
\end{eqnarray*}
By Assumption A1, for each $s\in [1,\hat{n}_{k}]$ and $ l\in [1,\hat{n}_{is}]$,  there is a $\rho_{k,1}>0$ such that
\begin{eqnarray*}
 &&P ((w_{t+1}, v_{t+1},u_t  )\in  D^k_{is,l}(2) |   \mathcal{F}_{t} )\\
&=&P ((w_1, v_1 )\in  B( (w^*_{is,l},v^*_{is,l}), d_{k})) I_{\{u_t\in \mathcal{U}^{i s}_l\}} \geq \rho_{k,1}I_{\{u_t\in \mathcal{U}^{i s}_l\}}  \quad \mbox{a.s.},
\end{eqnarray*}
and hence, by the independence of $\chi_t$ and $\psi_t$,   \dref{WV>0} indicates that for some $\rho_{k,2}>0$,
\begin{eqnarray}\label{PGamma2}
&&P ((  \chi_{t}, \psi_{t},  w_{t+1}, v_{t+1}, u_t )\in \Gamma^k_{i}|\mathcal{F}_{t-m}  )I_{\Omega_{t-m}}\nonumber\\
&\geq & \rho_{k,1}\sum_{s=1}^{\hat{n}_{k}}\sum_{l=1}^{\hat{n}_{is}}P (   \{ (\chi_{t}, \psi_{t})   \in   D^k_{is}(1)\}\cap \{u_t\in \mathcal{U}^{i s}_l\}
       |\mathcal{F}_{t-m}    )I_{\Omega_{t-m}}\nonumber\\
&=& \rho_{k,1}\sum_{s=1}^{\hat{n}_{k}}P (  (\chi_{t}, \psi_{t})   \in   D^k_{is}(1)
       |\mathcal{F}_{t-m}    )I_{\Omega_{t-m}}\nonumber\\
       &\geq & \rho_{k,1}\sum_{s=1}^{\hat{n}_{k}}
E ( I_{\{\chi_{t}   \in  B(\eta_{is},d_k) \backslash \bigcup\nolimits_{j=1}^{s-1}    B(\eta_{ij}, d_{k})  )  \}}P ( \psi_{t}   \in  B(\psi_{is},d_k)|\mathcal{F}^{\chi}_t )|\mathcal{F}_{t-m} )I_{\Omega_{t-m}}\nonumber\\
      &\geq & \rho_{k,1} \rho_{k,2}    P(\chi_{t}\in \bigcup\nolimits_{s\in[1,\hat{n}_{k}]}B(\eta_{is}, d_{k}) |\mathcal{F}_{t-m} )I_{\Omega_{t-m}}
   ,\quad \mbox{a.s.},
\end{eqnarray}
where $\mathcal{F}^{\chi}_t\triangleq \sigma\{\mathcal{F}_{t-m} \cup \sigma\{\chi_t\}   \}, t\geq m $. So, $\mathcal{F}_{t-m}\subset \mathcal{F}^{\chi}_t$.

%We now proceed to estimate the last term in \dref{PGamma2}.
Now, at time $t\geq m$,
take  $\{s_{t_j}\}_{j\in [1,m]}$  in Lemma \ref{stiFt}, which   corresponds to some random index $s_t$ and
 point $\eta^k_{is_t}=(e^{ik}_{s_{t_1},1},\ldots,e^{ik}_{s_{t_m},m})^T$   taking values   in  $ \{1,\ldots, \hat{n}_k\}$ and   $\Delta^k_i$ on set $\Omega_{t-m}$, respectively. %$\Omega^{ik}_{t-m,m-1}\in \mathcal{F}_{t-1}$.
Let $\bar d_{k}=d_{k}/\sqrt{m}<\sigma_k$ and
\begin{eqnarray}
\left\{
\begin{array}{ll}
\Omega^{ik}_{t,m}=\Omega_{t-m}\\
% \Omega^{ik}_{t,j}\triangleq {\{x_{t-l} \in B(e^{ik}_{s_{t_{l+1}},l+1}, \bar d_{k}), l\in [j,m] \}}\cap {\Omega_{t-m}}, & j\in [1,m-1]
\Omega^{ik}_{t,j}\triangleq {\{x_{t-j} \in B(e^{ik}_{s_{t_{j+1}},j+1}, \bar d_{k}) \}}\cap \Omega^{ik}_{t,j+1}, & j\in [1,m-1]
\end{array}.
\right.
\end{eqnarray}
%\begin{eqnarray}
%\left\{
%\begin{array}{ll}
%\Omega^{ik}_{t,m}=\Omega_{t-m}\\
%\Omega^{ik}_{t,j}\triangleq {\{x_{t-j} \in B(e^{ik}_{s_{t_{j+1}},j+1}, \bar d_{k}) \}}\cap\{ v_{t-j} \in B(v_{s_{t_{j+1}},j+1}, \bar d_{}) \}   \cap \Omega^{ik}_{t,j+1}, & j\in [1,m-1]
%\end{array}.
%%\right.
%\end{eqnarray}
According to Lemma \ref{stiFt}, $\Omega^{ik}_{t,j}$ is $\mathcal{F}_{t-j} $ measurable, $j\in [1,m]$.
    %and
 %\in (0, \min\{d'_k, \sigma_k/2\})$%\subset (0,\epsilon) be a sufficiently small constant fulfilling $\prod_{j=1}^m B(z_{j}, \bar d'_k)\subset B(\eta_{is_{t-m}},d'_k)$.
% Clearly, $\eta_{is_{t-1}}\in \Lambda_i$.
%$$ M^k_{tj}\triangleq \{x_{t-l+1}\in B(e^{ik}_{s_{t-m},l}, \bar d_{k}), l\in [j,m]\}\cap \Omega_{t-m}\in \mathcal{F}_{t-j+1}, \quad j\in [1,m].$$
So, by Assumption A1 and Lemma \ref{stiFt},   for any $t\geq m$, there is a $\rho_{k,3}>0$ such that
\begin{eqnarray*}
&&P(\chi_{t}\in \bigcup\nolimits_{s\in[1,\hat{n}_{k}]}B(\eta_{is}, d_{k})|\mathcal{F}_{t-m})I_{\Omega_{t-m}}%I_{\{\|\chi_{t-m}\|\leq C\}}I_{\{ \psi_{t-j}\in B(0,d'_k), j\in [1,m-1] \}}
\\
&\geq&P(\chi_{t}\in B(\eta_{is_{t}}, d_{k})|\mathcal{F}_{t-m}) I_{\Omega_{t-m}}\\&
\geq& P(\{x_{t} \in B(e^{ik}_{s_{t_1},1}, \bar d_{k})\}\cap   \Omega^{ik}_{t,1}|\mathcal{F}_{t-m})\\
&= &E( P( w_{t}\in B(e^{ik}_{s_{t_1},1}-f_{t-1}, \bar d_{k})   |   \mathcal{F}_{t-1}   )              I_{ \Omega^{ik}_{t,1}} |\mathcal{F}_{t-m})\\
&\geq & \rho_{k,3}P( \Omega^{ik}_{t,1}|\mathcal{F}_{t-m})\\&
=& \rho_{k,3}P(\{x_{t-1} \in B(e^{ik}_{s_{t_2},2}, \bar d_{k})\}\cap   \Omega^{ik}_{t,2}|\mathcal{F}_{t-m})\geq \cdots\geq \rho^m_{k,3} I_{\Omega_{t-m}},\quad \mbox{a.s.},
\end{eqnarray*}
where the third inequality follows from  \dref{ftjC}.
So, in view of \dref{PGamma2}, for each $t\geq m$,
\begin{eqnarray}\label{conditionalP}
P((  \chi_{t}, \psi_{t},  w_{t+1}, v_{t+1}, u_t )\in \Gamma^k_{i} |\mathcal{F}_{t-m})I_{\Omega_{t-m}}
\geq \rho_{k,1}\rho_{k,2}\rho^m_{k,3}I_{\Omega_{t-m}},\quad  \mbox{a.s.}.
\end{eqnarray}

%So, $\mathcal{F}^{\chi}_t\subset \mathcal{F}_t$ and hence \dref{conditionalP}
%\begin{eqnarray}\label{conditionalF}
%&&P\left((  \chi_{t}, \psi_{t},  w_{t+1}, v_{t+1} )\in \Gamma^k_{i} |\mathcal{F}^{\chi}_{t-m}\right)I_{\{\|\chi_{t-m}\|\leq C \}  }\nonumber\\
%&=& E\left(P\left((  \chi_{t}, \psi_{t},  w_{t+1}, v_{t+1} )\in \Gamma^k_{i} |\mathcal{F}_{t-m}\right)I_{\{\|\chi_{t-m}\|\leq C \}  }|\mathcal{F}^{\chi}_{t-m}\right)\nonumber\\
%&\geq& \rho_{k,1}\rho^m_{k,2}\rho_{k,3}P(v_{t-j}\in B(0,\bar d'_k), j\in [m,2(m-1)]|\mathcal{F}^{\chi}_{t-m})I_{\{\|\chi_{t-m}\|\leq C \}  }\nonumber\\
%&=& \rho_{k,1}\rho^m_{k,2}\rho_{k,3}\rho_{k,4}I_{\{\|\chi_{t-m}\|\leq C \}},\quad  \mbox{a.s.}.
%\end{eqnarray}

Now, for $t\geq 1$ and $l\in [0,m]$, denote $\zeta_{t,l}\triangleq \chi_{ (m+1)(t-1)+l}$ and
$$\zeta'_{t,l}\triangleq (  \chi_{ (m+1)t-1+l}, \psi_{ (m+1)t-1+l},  w_{ (m+1)t+l}, v_{ (m+1)t+l},  u_{ (m+1)t-1+l}).$$
Clearly,     $\{( \zeta_{t,l},  \zeta'_{t,l} )\}_{t\geq 1 }$
is adapted to the filtration $\{ \mathcal{F}'_{t,l}\}_{t\geq 1 }$ with $ \mathcal{F}'_{t,l}\triangleq\mathcal{F}_{ (m+1)t+l}$.
Since the experiment is $C$-recurrent,
%\begin{eqnarray}\label{IpsiC}
$\sum_{t=m}^\infty I_{\Omega_{t-m}} =\infty$ almost surely.
%\end{eqnarray}
So, by \dref{conditionalP},
$$
\sum\nolimits_{l=0}^m P_l =\infty,\quad \mbox{a.s.}\quad \mbox{with}\quad   P_l\triangleq \sum\nolimits_{t=2}^{\infty} P\left( \|\zeta_{t,l}\|\leq C,  \zeta'_{t,l}\in \Gamma^{k}_{i}|\mathcal{F}'_{ t-1,l}\right),
$$
which means there at least exists some $l\in [0, m]$ such that $P_l=\infty$ a.s..
According to the Borel-Cantelli-L$\acute{e}$vy  theorem,
$$
P\left(\{(  \chi_{ t}, \psi_{ t},  w_{ t+1}, v_{ t+1}, u_t )\in \Gamma^{k}_{i}\}, \mbox{ i.o.}\right)=1.
$$
Since $t_{k-1}<\infty$ almost surely, it is obvious that
 for every $i\in [n_{k-1},n_k)$,
$$
P\left(\{(  \chi_{t_{k-1}+ t}, \psi_{t_{k-1}+ t},  w_{t_{k-1}+ t+1}, v_{ t_{k-1}+t+1}, u_{t_{k-1}+ t} )\in \Gamma^{k}_{i}\}, \mbox{ i.o.}\right)=1.
$$
The result follows immediately.
\end{proof}

\emph{Proof of Theorem \ref{idensta}:} It is a direct result of Lemmas  \ref{conthe} and \ref{inputsta}.

\emph{Proof of Theorem \ref{iden}:} Given
 $\a\in A_0$, since  $\underline d^m ( \mathcal{P}_\a(\epsilon )| \mathbb{R}^{pm} )> 1/C_w$ for some $\epsilon>0$,  a countable set $\Delta_\a=\prod_{j=1}^m E^{\a }_j \subset \mathcal{P}_\a(\epsilon )$ exists ($|\Delta_\a|=\aleph_0 $) and
  $\underline d (E^{\a }_j| \mathbb{R}^{p} )> 1/(C_w-\sigma_\a), \sigma_\a \in (0,C_w)$.
  If $\eta\in  \Delta_\a$, by \dref{uthege}, for any $u\in \overline{B(0,C_u)}$,   there are some  $\psi\in\mathcal{V}^m$ and   $\pi(u)\in \mathcal{W}\times \mathcal{V}$ such that  %if  $(x,\chi, \psi)\in B(x,\epsilon) \times B( \eta, \epsilon)\times B(0,\epsilon)$, then
\begin{eqnarray*}
%\left\{
%\begin{array}{l}
%h_m(x,\chi,\psi)\in B( \beta, \varepsilon)
 \mathcal{B} h'\left(\check{x}_\epsilon,\check{u}_\epsilon,\check{\eta}_\epsilon,  \check{\pi}_\epsilon(u)\right )\notin \mbox{Im}( \mathcal{B}\hat h)^\epsilon_{\check{x}'_\epsilon, \check{u}_\epsilon,  \check{\b}}\quad \mbox{with}\quad  \check{\b}=\bar h(\check{x}_\epsilon,\check{\eta}_\epsilon,\check{\psi}_\epsilon).
%\end{array}.
%\right.
\end{eqnarray*}
As a result, by \dref{BneqB} and \dref{Bhhh},
$$
h'\left(\check{x}_\epsilon,\check{u}_\epsilon,\check{\eta}_\epsilon,  \check{\pi}_\epsilon(u)\right )\cap ( \bigcup\nolimits_{ \psi\in \mathcal{V}^m,\pi\in \mathcal{W}\times \mathcal{V} }\hat  h(\check{x}'_\epsilon, \check{u}_\epsilon, \check{\b},  \check{\psi}_\epsilon, \check{\pi}_\epsilon))=\emptyset.
$$
Moreover, $\mbox{Im}(\hat h_{\check{x}'_\epsilon, \check{u}_\epsilon, \check{\b}})\subset  \bigcup_{ \psi\in \mathcal{V}^m,\pi\in \mathcal{W}\times \mathcal{V} }\hat  h(\check{x}'_\epsilon,\check{u}_\epsilon, \check{\b},  \check{\psi}_\epsilon, \check{\pi}_\epsilon)$, then it yields
$$h'\left(\check{x}_\epsilon,\check{u}_\epsilon,\check{\eta}_\epsilon,  \check{\pi}_\epsilon(u) \right )\cap  \mbox{Im}(\hat h_{\check{x}'_\epsilon, \check{u}_\epsilon, \check{\b}})=\emptyset.$$
So, a similar proof of  Lemma \ref{CoverA} shows that
Lemma \ref{CoverA} holds with $\mathcal{P}_\a$ replaced by $\mathcal{P}_\a(\epsilon)$,  $ \hat{n}_{k}= \aleph_0 $ and  $\mathcal{S}=\mathbb{R}^{mp}$ ($C_0=\infty$). Now, since any $(\chi_0, \{u_t\})$ can be viewed as a $C$-recurrent experiment with $C=\infty$ and Lemmas \ref{stiFt}--\ref{inputsta} are still true for $C=\infty$,  the result follows from  Lemma \ref{conthe}.

\section{ Implementable Algorithm}\label{ImA} The estimator  in Section \ref{nonest} is only theoretical valid, %and the convergence rate might be very slow in some situations. % Its computation load might be huge while its convergence rate is  very likely to be slow.
so we are going to develop an implementable nonlinear estimator here. For simplicity,
study the following basic  control system
\begin{equation}\label{ymodel}
y_{t+1}=f(\theta,\varphi_{t})+u_t+w_{t+1},\quad t\geq 1-m
\end{equation}
in an experiment $(\chi_0, \{u_t\})\in \mathcal{E}$,  % or  its specific form
%$
%f(\theta,u_t,\varphi_{t})=f_1(\theta,\varphi_{t})+f_2(u_t)
%$
%in an experiment $(\chi_0, \{u_t\})\in \mathcal{E}$.
%If $(\chi_0, \{u_t\})\in \mathcal{E}$, then
%where  $f(\theta,u_t,\varphi_{t})=f(\theta,\xi(\varphi_{t}),\varphi_{t})$. So, without loss of generality,  we consider our  problem for the following model
%\begin{equation}\label{ymodel}
%y_{t+1}=f(\theta,\varphi_{t})+w_{t+1},\quad t\geq 1-m,
%\end{equation}
where $\mathcal{E}$ is defined by \dref{E},
 $\theta\in\Theta\subset \mathbb{R}^n$, $u_t, y_t, w_t$ are scalars  and
$\varphi_t= (y_t,  \ldots, y_{t-m+1})^T$. Moreover,
 $f(x,z): \mathbb{R}^n\times \mathbb{R}^{m}\to \mathbb{R}$ is known and $\frac{\partial f(x,z)}{\partial x}$ exists. Both the above two functions are continuous. Assume

  \begin{description}

%\item[A1]  Parameter $\theta \in \Theta$, either deterministic or stochastic, is independent of noise $\{w_t\}$.

\item[B1]
 $\{w_t\} $ is an  i.i.d sequence with $E w_1=0$ and  $E |w_1|^\kappa<\infty, \kappa>4$. In addition,\\
 (i) if $C_w<\infty$,  $w_1$  satisfies \dref{w1Cw};\\
(ii) if $C_w=\infty$, then for every $C'>0$,
$$\inf\nolimits_{\|z\|\leq C'
   } P( w_1\in    B(z, \delta))>0,\quad \forall \delta>0.$$
\end{description}

\begin{remark}\label{rhoneq0}
Assumption B1  includes  a large class of familiar distributions, such as uniform distribution $U(-C_w,C_w)$ for finite $C_w$, as well as
  Gaussian distributions and t-distributions for $C_w=\infty$.

\end{remark}

\subsection{Grid Searching Estimator}\label{GSE}
Assumption B1 implies that $E w^2_1$ exists. Denote  $\sigma^2_w\triangleq E w^2_1$ and $\bar{\sigma}^2_{w}\triangleq E(w^2_1-\sigma^2_w)^2$.   Recall that $\Omega_{i}=\{\|\varphi_{i}\|\leq C\}$ for some given $C>0$ ($C$ can be taken $\infty$). Let $\gamma>0$ and  define
\begin{equation}\label{Omegabar}
\Omega_{i}(\gamma,C)\triangleq
\left\{
\begin{array}{ll}
 \Omega_{i-m},&C_w<\infty\\
  \Omega_{i-m} \cap\{\|\varphi_i\|\leq \gamma\} ,&C_w=\infty
\end{array}.
\right.
\end{equation}
Let $\eta_t(\gamma)\triangleq \sum_{i=1}^t \Omega_{i}(\gamma,C)$.
At time  $t\geq 2$, the grid searching  estimator is designed according to function $\hat G_t: \mathbb{R}^n\times \mathbb{R}\to \mathbb{R}$ defined below:
 \begin{eqnarray}\label{hatGt=0}
\qquad\hat G_{t}(x,x')\triangleq  \sum\nolimits_{i=1}^{t-1} (f(x, \varphi_{i})-y_{i+1}-u_i)^2I_{\Omega_{i}(\gamma,C)}-\eta_t(\gamma) x',\quad x\in \mathbb{R}^n, x'\in \mathbb{R}.\quad
\end{eqnarray}
%Denote $C_n(x,r)$ as the $n$-cube  with the center located at $x\in \mathbb{R}^n$ and the length being $r>0$.
%Let $\hat\Theta_t$ be the uncertain domain of the unknown parameter $\theta$ at time $t\geq 1$.
Moreover,  we remark that the knowledge of $ \sigma^2_\omega$ can be described    by one of  the  following three scenarios:\\
(i)  $ \sigma^2_\omega$  is  known. Let   $\Sigma^0_t\equiv  \{\sigma^2_w\}$,  $t\geq 1$.
\\
(ii)  $ \sigma^2_\omega$   is unknown without any prior information.  Let $\Sigma^0_t= [0,t]$,  $t\geq 1$.\\
(iii)  $ \sigma^2_\omega$  is unknown but bounded by a known constant $\sigma>0$, i.e.,   $\sigma^2_w \leq\sigma$. Let $\Sigma^0_t\equiv [0,\sigma]$,  $t\geq 1$.\\
Let $\lambda, \gamma,C>0$ be some adjustable parameters and let
 $$
 C_\phi\triangleq\dfrac{ n\max_{x\in \Theta,\|z\|\leq \gamma}\| \frac{\partial f(x,z)}{\partial x}\|^2}{4}+1.
 $$%in the  algorithm to control the estimation accuracy and the computation load.      Also let. %Denote
%\begin{eqnarray}\label{Cf'}
% C_\phi\triangleq\max_{x\in \Theta}\max_{\|z\|\leq C}\left\| \dfrac{\partial f(x,z)}{\partial x}\right\|^2.
%\end{eqnarray}

\noindent\textbf{Algorithm}

Step 1: At time $t=0$, denote $o_{ 0}$ and $\sigma^2_{ 0}$ as the center points of sets $ \Theta$ and $\Sigma^0_{0}$, respectively. Set
\begin{equation}\label{the't0}
\hat \theta_{ 0}=o_{0}\quad \mbox{and}\quad \hat\sigma^2_{ 0}=\sigma^2_{0}.%\quad \mbox{and}\quad
%\hat \Theta_{ 0+1}=\Theta^0,\quad \hat \Sigma_{ 0+1}=\Sigma^0_{ 0+1}.
\end{equation}

%$$
%\hat G_{t}(o_1,\sigma^2_1)
%$$

Step 2: At time $t\geq 1$, equally divide $\Theta$ and $\Sigma^0_t$ into two finite sequences of small boxes $\{\Theta_{ti}\}$ and $\{\Sigma_{tj}\}$    that $ \Theta=\bigcup_i \Theta_{ti}$ and $\Sigma^0_t=\bigcup_j \Sigma_{tj}$, where
the side lengthes of  $\Theta_{ti}$ and $\Sigma_{tj}$  are less than $1/(t^{\frac{1}{4}-\frac{1}{2\kappa}-\lambda})$ and $1/(t^{\frac{1}{2}-\frac{1}{\kappa}-\lambda})$, respectively. Let $o_{ ti} $    and   $\sigma^2_{ tj}$ be the center points of  $ \Theta_{ti}$ and  $\Sigma_{tj}$.   If
 $$
\mathcal{J}_t\triangleq\{(i,j): | \hat G_{t}(o_{ti},\sigma^2_{tj})   |\leq C_\phi t^{\frac{1}{2}+\frac{1}{\kappa}+2\lambda}  \}=\emptyset,$$
set% $\kappa_t=\kappa_{t-1}+\gamma$ and
\begin{equation}\label{It=I0}
\hat \theta_{t}=\hat \theta_{t-1}\quad \mbox{and}\quad\hat\sigma^2_{t}=\hat\sigma^2_{t-1}.
%\left\{
%\begin{array}{l}
%\hat \theta'_{t}=\hat \theta'_{t-1}, \quad \hat\sigma^2_{t}=\hat\sigma^2_{t-1}\\
%\hat \Theta_{ t+1}=\Theta^0,\quad \hat \Sigma_{ t+1}=\Sigma^0_{ t+1}
%\end{array}.
%\right.
\end{equation}
Otherwise, for $\mathcal{J}_t\neq \emptyset$, take an arbitrary $(i^*,j^*)\in \mathcal{J}_t$ satisfying
\begin{equation}\label{ij*}
 (i^*,j^*)\in\{(i,j)\in \mathcal{J}_t: \sigma^2_{tj^*}=\min\nolimits_{(i,j)\in \mathcal{J}_t}\sigma^2_{tj}\}.
 \end{equation}
 Set %$\kappa_t=\kappa_{t-1}$ and
\begin{eqnarray}\label{the'}
\hat \theta_{t}=o_{ti^*}\quad\mbox{and}\quad  \hat\sigma^2_{t}=\sigma^2_{tj^*}.
%\left\{
%\begin{array}{l}
%\hat \theta'_{t}=o_{ti^*},\quad  \hat\sigma^2_{t}=\sigma^2_{tj^*}\\
%\hat \Theta_{t+1}=C_n(o_{ti^*},\kappa_t\sqrt[4]{\lambda^2\log t/t^{2\delta-1}})\cap \hat\Theta_t\\
%\hat \Sigma_{t+1}=C_n(\sigma^2_{tj^*},2\kappa_t\lambda\sqrt{\log t/t^{2\delta-1}})\cap \hat\Sigma_t
%\end{array}.
%\right.
\end{eqnarray}

\subsection{Strong Consistency} For $1\leq k\leq n$, let $ x^{(k)}, \bar x^{(k)} \in \mathbb{R}^{2^{k-1}n}$,
$ y^{(k)}, \bar y^{(k)} \in \mathbb{R}^{2^{k-1}m}$ and
$z^{(k)}=\mbox{col}\{x^{(k)},y^{(k)}\}$. Write $x^{(1)}=(x^{(1)}_1,\ldots,x^{(1)}_n)$. Now,
recursively define a sequence of functions  $\{ g^{(k)}_j, 1\leq k\leq j\leq n\}$ for system  \dref{ymodel}  as follows:
 \begin{eqnarray}\label{gk}
\left\{
 \begin{array}{ll}
 g^{1}_j(x^{(1)}, y^{(1)})\triangleq\frac{\partial f(x^{(1)}, y^{(1)})}{\partial x^{(1)}_j}, & 1\leq j\leq n\\
g^{k+1}_j(z^{(k)},\bar z^{(k)})\triangleq g^{k}_{k}( z^{(k)})g^{k}_{j}( \bar z^{(k)})-g^{k}_{k}(\bar z^{(k)})g^{k}_{j}(z^{(k)}),&1\leq k< j\leq n
\end{array}.\quad
\right.
\end{eqnarray}
% where
%\begin{eqnarray}\label{gk}
%\left\{
%\begin{array}{ll}
%g^{1}_j(x, z^{(1)})\triangleq\dfrac{\partial f(x, z^{(1)})}{\partial x_j}, &  1\leq j\leq n\\
%g^{k+1}_j(x, z^{(k)},\bar z^{(k)})\triangleq \left(g^{k}_{k}(x,  z^{(k)})g^{k}_{j}(x,  \bar z^{(k)})-g^{k}_{k}(x, \bar z^{(k)})g^{k}_{j}(x,z^{(k)})\right)^2,&  1\leq k< j\leq n
%\end{array},
%\right.
%\end{eqnarray}
%$x%=(x_1,x_2,\ldots,x_n)^T\in \mathbb{R}^n$,
%$z^{(k)}=\mbox{col}\{x^{(k)},y^{(k)}\}$ with
%$ x^{(k)}, \bar x^{(k)} \in \mathbb{R}^{2^{k-1}n}$
%$ y^{(k)}, \bar y^{(k)} \in \mathbb{R}^{2^{k-1}m}$, $1\leq k\leq n$.
% \begin{eqnarray}\label{g1}
%\end{eqnarray}
%Denote $X^{k}\triangleq\prod_{i=1}^{2^{k-1}m }X$ for a set $X$. Then,
%given a $C_\varphi>0$, define
%\begin{equation}\label{Bk}
%\mathcal{B}^k_{\varphi,w}\triangleq (-C_w+C_f, C_w-C_f)^{2^{k-1}m},\quad k\in[1,n],
%\end{equation}

\begin{example}\label{exg}
In system \dref{ymodel} with $n=1$, $g^1_1(x,y)=\frac{\partial f(x, y)}{\partial x}$. For $n=2$,
$$
g^2_2(x_1, x_2, \bar x_1, \bar  x_2; y, \bar y)=\frac{\partial f(x_1,x_2, y)}{\partial x_1} \frac{\partial  f(\bar x_1, \bar x_2, \bar y)}{\partial \bar x_2}-\frac{\partial f(\bar x_1, \bar x_2, \bar y)}{\partial  \bar x_1}\frac{\partial f(x_1,x_2, y)}{\partial x_2}.
$$
\end{example}

%Now, we present the main theorems of this section.

The convergence  of estimates  $\hat \theta_{t}$ is related to the density of set
\begin{equation}\label{P*x}
\mathcal{P}'\triangleq\{\b\in \mathbb{R}^{2^{n-1}m}:g^n_n(x,\b) \neq0, \forall x\in \Theta^{2^{n-1}} \}
\end{equation}
in $ \mathbb{R}^{2^{n-1}m}$ for $C_w=\infty$ or
in $\mathcal{S}=\overline {B(0,C_0)}\subset \mathbb{R}^{2^{n-1}m}$  for $C_w<\infty$, where  $C_0$ is defined similarly as that in \dref{ff'<}.
This claim is verified for the case where
the closed-loop system is stable, i.e.,
$$
\sup_{t\geq 1} \frac{1}{t}\sum_{i=1}^t y^2_i<\infty,\quad \mbox{a.s.}.
$$

\begin{theorem}\label{the}
 Under Assumption B1,  let the  closed-loop system \dref{ymodel} be  stable.
If for each $x\in \Theta$,  either
 $\underline{d}^m(\mathcal{P}'|\mathcal{S}^{2^{n-1}m})>1/C_w$ for $C_w<\infty$ or
 $\mathcal{P}' \neq \emptyset$ for $C_w=\infty$,   then   by choosing parameter $\gamma$   sufficiently large and parameter $\lambda\in (0,\frac{1}{4}-\frac{1}{2\kappa})$,   the grid searching estimator satisfies
\begin{eqnarray}\label{the'err}
\|\tilde{\theta}_{t}\|=O\left(\dfrac{1}{t^{\frac{1}{4}-\frac{1}{2\kappa}-\lambda}}\right)\rightarrow 0,\quad \mbox{a.s.}.
\end{eqnarray}
\end{theorem}
% A  typical experiment is that the corresponding closed-loop  system \dref{ymodel} is stable, i.e.,
%$
%\sup_{t\geq 1} \frac{1}{t}\sum_{i=1}^t y^2_i<\infty
%$ almost surely.
%Choosing  a sufficiently large $C$ then yields    $t=O(\eta_t)$ a.s.  and   $P(\Omega_\eta)=1$. The unknown parameter thus can be identified.    %if
%\begin{description}
%\item[B1']
% The noise $\{w_t\} $ is an  i.i.d sequence with $w_1$ being Gaussian distributed.
%\end{description}

\begin{example}
Let us consider system \dref{ymodel} with $f(x_1,x_2,y)=x_1y^{b_1}+x_2 y^{b_2}, $ where $x_1,x_2,y \in \mathbb{R}$ and $b_1\neq b_2$. By Example \ref{exg}, $g^2_2(x_1, x_2, \bar x_1, \bar  x_2; y, \bar y)=y^{b_1} \bar y^{b_2}-\bar y^{b_1}  y^{b_2}, $ which causes $\mathcal{P}'$  dense in $\mathbb{R}^{2}$.
\end{example}

\begin{example}
If $C_w=\infty$, the only requirement on $\mathcal{P}' $ for parameter identifiability  is $\mathcal{P}' \neq\emptyset$. This applies to a lot of control systems. For instance, in  system \dref{ymodel}, let $f(x,y)=\sin (xy) $ for $x,y \in \mathbb{R}$ and $\Theta=[0,2\pi]$. Example \ref{exg} shows $g^n_n(x,y)=\cos(xy)$.  If $y=1/8$, then $\cos(xy)\in [\sqrt{2}/2,1]$ for all $x\in [0,2\pi]$.  Thus, $1/8\in \mathcal{P}' $.
\end{example}

\subsection{Proof of Theorem \ref{the}} We first introduce some notations. For two vectors $p=(p_i)_{i=1}^{l}, q=(q_i)_{i=1}^{l}, l\geq 1$, we say $ p\prec q$ if there is an index $ j\in [1,l)$ such that $p_i=q_i,1\leq i\leq j$ and $p_{j+1}<q_{j+1}$.
Define a series of  sets $\{\mathcal{H}^t_k\}$ by
\begin{equation}
\mathcal{H}^t_k
\triangleq\left\{
\begin{array}{ll}
\{1,2,\ldots,t\},& k=1\\
\{(p,q):p,q\in\mathcal{H}^t_{k-1}, p\prec q\},& k\in [2,n]
\end{array}.
\right.
\end{equation}
%where $ p\prec q$ means $p_i=q_i,1\leq i\leq j<2^{k-2}$ and $p_{j+1}=q_{j+1}$ for $p=(p_i)_{i=1}^{2^{k-2}}, q=(q_i)_{i=1}^{2^{k-2}}$.

\begin{lemma}\label{sumaa}
Let  $\a_i\triangleq (a_{i,1},\ldots,a_{i,n})^T$, $i\in [1,t]$ for some fixed  $t\geq 1$ and $n\geq 1$. Denote $M(k)$ as the $k$th order leading principal minor of $\det(\sum_{i=1}^t \a_i\a^T_i)$ for $k\in [1,n]$ and $M'(k,k)$ as the $k,k$ cofactor of $M(k+1) $ for $k\in [1,n-1]$. If $\sum_{i=1}^t a^2_{i,j}\neq 0 $ for all $j\in [1,n]$, then there is a sequence $\{\mu_h(k),\nu_h(k), h\in \mathcal{H}^t_k, k\in [1,n]\}$ such that
 each $M(k)$ and $M'(k,k)$ can be written as \footnote{$\prod_{j=1}^{0}(\sum_{h\in \mathcal{H}^t_j} \mu^2_h(j))^{-1}\triangleq 1.$}
\begin{equation}\label{MM'}
\left\{
\begin{array}{ll}
M(k)= \dfrac{ \sum_{h\in \mathcal{H}^t_k} \mu^2_h(k)}{ \prod_{j=1}^{k-1}   (\sum_{h\in \mathcal{H}^t_j} \mu^2_h(j))^{k-j-1}},&k\in [1,n]\\
 M'(k,k)=\dfrac{\sum_{h\in \mathcal{H}^t_k} \nu^2_h(k)}{\prod_{j=1}^{k-1}   (\sum_{h\in \mathcal{H}^t_j} \mu^2_h(j))^{k-j-1}},&k\in [1,n-1]
\end{array},
\right.
\end{equation}
%\begin{equation}\label{MM'}
%M(k)= \dfrac{ \sum_{h\in \mathcal{H}^t_k} \mu^2_h(k)}{(\sum_{i=1}^t a^2_{i,1})^{k-2}},k\in [1,n]\quad M'(k,k)=\dfrac{\sum_{h\in \mathcal{H}^t_k} \nu^2_h(k)}{(\sum_{i=1}^t a^2_{i,1})^{k-2}},k\in [1,n-1]
%\end{equation}
where, for each $h=(p,q)\in \mathcal{H}^t_{k}, k\in [2,n]$,
\begin{equation}\label{muk+1}
\mu_h(k)=\mu_p(k-1)\nu_q(k-1)-\mu_q(k-1)\nu_p(k-1)
\end{equation}
and there is a function $\zeta_{h,k}(\cdot):\mathbb{R}^{tk}\rightarrow \mathbb{R}$  independent of  $\a_i, i\in [1,t]$  such that
\begin{equation}\label{zetahk}
\left\{
\begin{array}{ll}
\mu_h(k)=\zeta_{h,k}(a_{i,j},i=1,\ldots,t,j=1,\ldots,k)\\
\nu_h(k)=\zeta_{h,k}(a_{i,j},i=1,\ldots,t,j=1,\ldots,k-1,k+1)
\end{array},\quad k\geq 1.
\right.
\end{equation}
%Here,  functions $\zeta_{h,k}$ are all independent of the values of $a_{i,j}, i\in [1,t], j\in [ 1,n]$.
\end{lemma}

\begin{proof} Let $n=2$.
Clearly,  $M(1)=\sum_{i=1}^t a^2_{i,1}, M'(1,1)=\sum_{i=1}^t a^2_{i,2}$ and
\begin{eqnarray*}
M(2)&=&\left(\sum_{i=1}^t a^2_{i,1}\right)\left(\sum_{i=1}^t a^2_{i,2}\right)-\left(\sum_{i=1}^ta_{i,1} a_{i,2}\right)^2,\\
&=& \sum\nolimits_{(p,q)\in \mathcal{H}^t_{2}}\left(a_{p,1} a_{q,2}-a_{q,1} a_{p,2}\right)^2.
\end{eqnarray*}
Similarly, $M'(2,2)=\sum_{(p,q)\in \mathcal{H}^t_{2}}\left(a_{p,1} a_{q,3}-a_{q,1} a_{p,3}\right)^2.$
Hence,
the lemma is true when $n=2$ with $\mu_h(1)=a_{h,1},\nu_h(1)=a_{h,2}, h\in \mathcal{H}^t_1$ and
\begin{equation}\label{zetah2}
\left\{
\begin{array}{ll}
\mu_h(2)=a_{p,1} a_{q,2}-a_{q,1} a_{p,2}\\
\nu_h(2)=a_{p,1} a_{q,3}-a_{q,1} a_{p,3}
\end{array},\quad h=(p,q)\in \mathcal{H}^t_2.
\right.
\end{equation}
%$$
%\mu_h(2)=a_{p,1} a_{q,2}-a_{q,1} a_{p,2},\quad \nu_h(2)=a_{p,1} a_{q,3}-a_{q,1} a_{p,3},\quad h=(p,q) \mbox{ with } p,q\in \mathcal{H}_1.
%$$
%If we denote $\beta_h(k)\triangleq (\mu_h(k), \nu_h(k))^T$, then,
 %$$M(2)=\det\left(\sum_{h\in  \mathcal{H}_1} \beta_h(1)\beta^T_h(1)\right).$$

 Now, let $n\geq 3$. Suppose for some integer $l\in [2, n-1]$, there is a sequence  $\{\mu_h(k), \nu_h(k), h\in  \mathcal{H}^t_k, k\in[1,l]\}$ satisfying     \dref{MM'}--\dref{zetahk},
 % and
% \begin{equation}\label{Mk}
%M(k+1)=\dfrac{ 1}{\prod_{j=1}^{k}   (\sum_{h\in \mathcal{H}_j} \mu^2_h(j))^{k-j}}\det\left(\sum_{h\in  \mathcal{H}_{k}} \beta_h(k)\beta_h^T(k)\right).
 %\end{equation}
%Note that \dref{Mk} equals to \dref{muk+1} due to
%\begin{eqnarray*}
%\det\left(\sum_{h\in  \mathcal{H}_{k}} \beta_h(k)\beta_h^T(k)\right)&=&\left(\sum_{h\in  \mathcal{H}_{k}} \mu^2_h(k)\right)\left(\sum_{h\in  \mathcal{H}_{k}} \nu^2_h(k)\right)-\left(\sum_{h\in  \mathcal{H}_{k}} \mu_h(k)\nu_h(k)\right)^2\\
%&=& \sum_{p,q\in \mathcal{H}_{k}}\left(\mu_p(k)\nu_q(k)-\mu_q(k)\nu_p(k)\right)^2
%\end{eqnarray*}
then we  will show the existence of $\{\mu_h(k), \nu_h(k), h\in  \mathcal{H}^t_k, k\in[1,l+1]\}$ such that \dref{MM'}--\dref{zetahk} hold.

For $k=l+1$, write  $M(k)$ as a  block matrix by
\begin{eqnarray}
\left|
\begin{array}{ll}
\sum_{i=1}^t a^2_{i,1} & M^T_1(k)\\
M_1(k)& M_2(k)
\end{array}
\right|,
\end{eqnarray}
where $$M_1(k)=\sum_{i=1}^t(a_{i,1}a_{i,2},\ldots, a_{i,1}a_{i,k})^T \,\,\,  \mbox{and}\,\,\,    M_2(k)=\sum\limits_{i=1}^t (a_{i,2},\ldots, a_{i,k})^T(a_{i,2},\ldots, a_{i,k}). $$         Since  $\sum_{i=1}^t a^2_{i,1}\neq 0$,  then
\begin{eqnarray*}
M(k)&=&\left(\sum_{i=1}^t a^2_{i,1}\right)\det \left(M_2(k)-\dfrac{M_1(k)M^T_1(k)}{\sum_{i=1}^t a^2_{i,1}} \right)\\
&=&  \dfrac{\det \left(M_2(k)\left(\sum_{i=1}^t a^2_{i,1}\right)-M_1(k)M^T_1(k) \right)}{(\sum_{i=1}^t a^2_{i,1})^{k-2}}.
\end{eqnarray*}
Note that the $(j,s)$ entry of  $M_2(k)\left(\sum_{i=1}^t a^2_{i,1}\right)-M_1(k)M^T_1(k)$ is
\begin{eqnarray*}
&&\left(\sum_{i=1}^t a_{i,(j+1)}a_{i,(s+1)}\right)    \left(\sum_{i=1}^t a^2_{i,1}\right)-\left(\sum_{i=1}^t a_{i,1} a_{i,(j+1)}\right)\left(\sum_{i=1}^t a_{i,1} a_{i,(s+1)}\right)\\
&=&\sum_{(p,q)\in \mathcal{H}^t_{2}} (a_{p,1}a_{q,(s+1)}-a_{q,1}a_{p,(s+1)})(a_{p,1}a_{q,(j+1)}-a_{q,1}a_{p,(j+1)}) ,\quad 1\leq s,j\leq k-1.
\end{eqnarray*}
Let $\a'_{p,q}(k)\triangleq (a_{p,1}a_{q,2}-a_{q,1}a_{p,2},\ldots,a_{p,1}a_{q,k}-a_{q,1}a_{p,k})^T$,
then
 \begin{equation}\label{Mj+1}
M(k)=\dfrac{ \det\left(\sum_{(p,q)\in \mathcal{H}^t_2} \a'_{p,q}(k) (\a'_{p,q}(k))^T \right)}  {(\sum_{i=1}^t a^2_{i,1})^{k-2}},\quad k=l+1.
 \end{equation}
Observe that matrix $\sum_{(p,q)\in \mathcal{H}^t_2} \a'_{p,q}(l+1) (\a'_{p,q}(l+1))^T $ has the same form of $M(l)$, which is of dimension $l$. Moreover, $a_{i,j},j\in [1,l]$ can be taken any values in $M(l)$,  % $\{\mu_h(k), \nu_h(k), h\in  \mathcal{H}^t_k\}$ for  $ k\in[1,l]$,
so
 by the assumption and \dref{zetah2},
\begin{equation}\label{a'a'}
\det\left(\sum_{(p,q)\in \mathcal{H}^t_2} \a'_{p,q}(l+1) (\a'_{p,q}(l+1))^T \right)    =\dfrac{\sum_{h\in  \mathcal{H}^t_{l+1}}\mu'^2_h(l)}{\prod_{j=1}^{l-1}   (\sum_{h\in \mathcal{H}^t_{j+1}} \mu'^2_h(j))^{l-j-1}}
 \end{equation}
holds for some  $\{\mu'_h(k), \nu'_h(k), h\in  \mathcal{H}^t_{k+1}, k\in[1,l]\}$ satisfying
\begin{eqnarray}\label{m'k}
\left\{
\begin{array}{ll}
\mu'_h(1)=\mu_h(2), \nu'_h(1)=\nu_h(2), & h\in \mathcal{H}^t_2 \\
\mu'_h(k+1)=\mu'_{p}(k)\nu'_{q}(k)-\mu'_{q}(k)\nu'_{p}(k),& h=(p,q)\in \mathcal{H}^t_{k+2}, k\geq1
\end{array}.
\right.
\end{eqnarray}
In addition, there is a sequence of  $\{\zeta'_{h,k}(\cdot),  h\in \mathcal{H}^t_{k+1},k\in[1,l]\} $ such that
\begin{eqnarray}\label{mn'k}
\left\{
\begin{array}{l}
\mu'_h(k)=\zeta'_{h,k}(  a_{p,1}a_{q,j}-a_{q,1}a_{p,j},(p,q)\in \mathcal{H}^t_2,j=2,\ldots,k+1)\\
\nu'_h(k)=\zeta'_{h,k}(  a_{p,1}a_{q,j}-a_{q,1}a_{p,j},(p,q)\in \mathcal{H}^t_2,j=2,\ldots,k, k+2)\\
\end{array}.
\right.
\end{eqnarray}
%$\mu'_h(k)=\mu'_{p'}(k-1)\nu'_{q'}(k-1)-\mu'_{q'}(k-1)\nu'_{p'}(k-1), h=(p',q')$ for $p',q'\in \mathcal{H}_{k} $.

Considering \dref{m'k}, if $l=2$, then for all $k\in [1,l-1]$,
\begin{equation}\label{mu'=mu}
\mu'_h(k)=\mu_h(k+1) \quad\mbox{and} \quad \nu'_h(k)=\nu_h(k+1),\quad  h\in \mathcal{H}^t_{k+1}.
 \end{equation}
For $l>2$,  suppose  there is an  $s\in [1,l-2]$ such that  \dref{mu'=mu} holds       for all $k\in[1,s]$. Since $s+2\leq l$,
 then by \dref{muk+1} and \dref{m'k}, for any $ h=(p,q)\in \mathcal{H}^t_{s+2}$,
$$
\mu'_h(s+1)=\mu_{p}(s+1)\nu_{q}(s+1)-\mu_{q}(s+1)\nu_{p}(s+1)=\mu_h(s+2).
$$
This, together with \dref{zetahk} and \dref{mn'k}, infers
\begin{eqnarray*}
\mu_h(s+2)&=&\zeta_{h,s+2}(a_{i,j},i=1,\ldots,t,j=1,\ldots,s+2)\\
&=&\mu'_h(s+1)=\zeta'_{h,s+1}(  a_{p,1}a_{q,j}-a_{q,1}a_{p,j},(p,q)\in \mathcal{H}^t_2,j=2,\ldots,s+2).
\end{eqnarray*}
Note that $\zeta_{h,s+2}$ and $\zeta'_{h,s+1}$ are independent of the values of $\a_i, i\in[1,t]$,   then
\begin{eqnarray*}
&&\zeta_{h,s+2}(a_{i,j},a_{i,s+3},i=1,\ldots,t,j=1,\ldots,s+1) \\
&=&  \zeta'_{h,s+1}(  a_{p,1}a_{q,j}-a_{q,1}a_{p,j}, (p,q)\in \mathcal{H}^t_2,j=2,\ldots,s+1, s+3).
\end{eqnarray*}
Or equivalently, $\nu'_h(s+1)=\nu_h(s+2)$. Therefore,  $\mu'_h(k)=\mu_h(k+1)$ and  $\nu'_h(k)=\nu_h(k+1)$ for all $k\in [1,l-1]$.

Define $\mu_h(l+1)\triangleq \mu'_h(l)$ for all $h\in  \mathcal{H}^t_{l+1}$, then
\begin{eqnarray*}
\mu_h(l+1)&=&\mu'_{p}(l-1)\nu'_{q}(l-1)-\mu'_{q}(l-1)\nu'_{p}(l-1) \\
&=&  \mu_{p}(l)\nu_{q}(l)-\mu_{q}(l)\nu_{p}(l),\quad h=(p,q)\in \mathcal{H}^t_{l+1}.
\end{eqnarray*}
Since  $\sum_{i=1}^t a^2_{i,1}=\sum_{h\in\mathcal{H}^t_1} \mu^2_h(1), $   combining \dref{Mj+1} and    \dref{a'a'} leads to the first formula of \dref{MM'} immediately for $k=l+1$.
If $l<n-1$, also let $\nu_h(l+1)\triangleq \nu'_h(l), h\in  \mathcal{H}^t_{l+1}$. Note that
\begin{eqnarray*}
M'(l+1,l+1)%&=&\dfrac{\sum_{h\in \mathcal{H}_{l+1} }\zeta_{h,l+1}(a_{i,j},a_{i,l+2},i=1,\ldots,t,j=1,\ldots,l)}{\prod_{j=1}^{l}   (\sum_{h\in \mathcal{H}_j} \mu^2_h(j))^{l-j}}\\
&=&\dfrac{\sum_{h\in \mathcal{H}^t_{l+1} }(\zeta'_{h,l}(a_{p,1}a_{q,j}-a_{q,1}a_{p,j},(p,q)\in \mathcal{H}^t_2,j=2,\ldots,l,l+2))^2}{\prod_{j=1}^{l}   (\sum_{h\in \mathcal{H}^t_j} \mu^2_h(j))^{l-j}}\\
&=& \dfrac{\sum_{h\in \mathcal{H}^t_{l+1}} \nu'^2_h(l)}{\prod_{j=1}^{l}( \sum_{h\in \mathcal{H}^t_j} \mu^2_h(j))^{l-j}}=\dfrac{\sum_{h\in \mathcal{H}^t_{l+1}} \nu^2_h(l+1)}{\prod_{j=1}^{l}( \sum_{h\in \mathcal{H}^t_j} \mu^2_h(j))^{l-j}}.
\end{eqnarray*}
Finally, for each  $h\in \mathcal{H}^t_{l+1}$, there is a $\zeta_{h,l+1}$ independent of $\a_i,i\in [1,t]$ such that
\begin{eqnarray*}
\mu_h(l+1)&=&\zeta'_{h,l}( a_{p,1}a_{q,j}-a_{q,1}a_{p,j},(p,q)\in \mathcal{H}^t_2,j=2,\ldots,l+1)\\
&=&\zeta_{h,l+1}(a_{i,j}, i=1,\ldots,t,j=1,\ldots,l+1)\\
\nu_h(l+1)&=&\zeta'_{h,l}( a_{p,1}a_{q,j}-a_{q,1}a_{p,j},(p,q)\in \mathcal{H}^t_2,j=2,\ldots,l, l+2)\\
&=&\zeta_{h,l+1}(a_{i,j}, i=1,\ldots,t,j=1,\ldots,l, l+2).
\end{eqnarray*}
So, with $\{\mu_h(l+1), \nu_h(l+1), h\in  \mathcal{H}^t_{l+1}\}$ defined above,  \dref{MM'}--\dref{zetahk} hold for $k=l+1$,
which completes the proof  by induction.
\end{proof}

%\begin{lemma}\label{nusum}
%Let  $\{\mu_h(k),  h\in  \mathcal{H}^t_{k}, k\in [1,n]\}$  and $\{\nu_h(k), h\in  \mathcal{H}^t_{k}, k\in [1,n-1]\}$   be defined by        Lemma \ref{sumaa}.
%Then,
%$$
% \nu^2_q(n-1) \left( \sum_{i=1}^t a^2_{i,n}  \right)=  \prod_{j=1}^{n-2}  \left (\sum_{h\in \mathcal{H}^t_j} \mu^2_h(j)\right).
%$$
%\end{lemma}

\begin{lemma}\label{detsumaa}
Let the conditions of Lemma \ref{sumaa} hold and
denote $\lambda_{\min}(\sum_{i=1}^t \a_i\a^T_i)$ as the minimal eigenvalue of matrix $\sum_{i=1}^t \a_i\a^T_i$.
If there is a number $\epsilon>0$ such that for each $k\in [1,n-1]$ and $s\in [k+1, n]$,
\begin{eqnarray}\label{munue}
\sum\nolimits_{p,q\in \mathcal{H}^t_k} ( \mu_p(k)\nu_{q, s}(k)-\mu_q(k)\nu_{p, s}(k))^2\geq 2\epsilon \sum\nolimits_{p,q\in \mathcal{H}^t_k}  \mu^2_p(k)\nu^2_{q, s}(k),
\end{eqnarray}
where $\nu_{h, s}(k)\triangleq\zeta_{h,k}(a_{i,j},i=1,\ldots,t;j=1,\ldots,k-1,s), h\in  \mathcal{H}^t_k, s\in [k,n]$\footnote{$ \nu_{h, k}(k)=\mu_h(k)$}, then
 $$\lambda_{\min}\left(\sum_{i=1}^t \a_i\a^T_i\right) \geq \frac{\epsilon^{n-1}}{n} \min_{j\in[1,n]}\sum_{i=1}^t a^2_{ij}.$$
\end{lemma}

\begin{proof}
Let $\pi(n-1)$ be the set of the $(n-1)$-permutations of $\{1,2,\ldots,n\}$. For $p=(i_1,\ldots,i_{n-1})\in \pi(n-1)$, define $\a_{i,p}\triangleq (a_{i,i_1},\ldots,a_{i,i_{n-1}})^T$ and  denote the $n$ eigenvalues of $\sum_{i=1}^t \a_i\a^T_i$ by $\lambda_i,1\leq i\leq n$ with $\lambda_i\geq \lambda_{i+1},1\leq i\leq n-1$. According to the  Vieta's formulas, one has
\begin{eqnarray}
&& \prod_{i=1}^{n} \lambda_{i}=\det\left( \sum_{i=1}^t \a_{i}\a^T_{i}\right)\\
&&\sum_{ (i_1,\ldots,i_{n-1} )\in \pi(n-1) } \prod_{j=1}^{n-1} \lambda_{i_j}=\sum_{ p\in \pi(n-1) }\det\left( \sum_{i=1}^t \a_{i,p}\a^T_{i,p}\right).
\end{eqnarray}
Note that  reordering the $n$  elements   $a_{i,1},\ldots,a_{i,n}$ of vector  $\a_i, i\in [1,t]$ does not change the  minimal eigenvalue of
$\sum_{i=1}^t \a_i\a^T_i$. So,
without loss of generality, for $p_1=(1,2,\ldots,n-1)$, assume
$$
\det\left( \sum_{i=1}^t \a_{i,p_1}\a^T_{i,p_1}\right)=\max_{p\in \pi(n-1)  }\det\left( \sum_{i=1}^t \a_{i,p}\a^T_{i,p}\right).
$$
Therefore,
\begin{eqnarray}
 \lambda_n&\geq& \dfrac{\prod_{i=1}^{n} \lambda_{i}}{\sum_{ (i_1,\ldots,i_{n-1} )\in \pi(n-1) }  \prod_{j=1}^{n-1} \lambda_{i_j}}\nonumber\\
&= &\dfrac{\det\left( \sum_{i=1}^t \a_{i}\a^T_{i}\right)}{   \sum_{ p\in \pi(n-1) }\det\left( \sum_{i=1}^t \a_{i,p}\a^T_{i,p}\right)}
\geq \dfrac{\det\left( \sum_{i=1}^t \a_{i}\a^T_{i}\right)}{  n \det\left( \sum_{i=1}^t \a_{i,p_1}\a^T_{i,p_1}\right)}.
\end{eqnarray}
Consequently, by Lemma \ref{sumaa} and \dref{munue},
\begin{eqnarray}\label{lambdana}
 \lambda_n
&\geq &\dfrac{1}{n}\dfrac{ \sum_{h\in \mathcal{H}^t_n} \mu^2_h(n)} {\sum_{h\in \mathcal{H}^t_{n-1}} \mu^2_h(n-1)   }  \dfrac{ \prod_{j=1}^{n-2}   (\sum_{h\in \mathcal{H}^t_j} \mu^2_h(j))^{n-j-2}   } { \prod_{j=1}^{n-1}   (\sum_{h\in \mathcal{H}^t_j} \mu^2_h(j))^{n-j-1}}\nonumber\\
&=&\dfrac{ \sum_{(p,q)\in \mathcal{H}^t_{n}} (\mu_p(n-1)\nu_q(n-1)-\mu_q(n-1)\nu_p(n-1))^2} {n  \prod_{j=1}^{n-1}   (\sum_{h\in \mathcal{H}^t_j} \mu^2_h(j)) }\nonumber\\
&=&\dfrac{ \sum_{p,q\in \mathcal{H}^t_{n-1}} (\mu_p(n-1)\nu_q(n-1)-\mu_q(n-1)\nu_p(n-1))^2} {2n  \prod_{j=1}^{n-1}   (\sum_{h\in \mathcal{H}^t_j} \mu^2_h(j)) }\nonumber\\
&\geq & \dfrac{ \epsilon \sum_{p,q\in \mathcal{H}^t_{n-1}}  \mu^2_p(n-1)\nu^2_q(n-1)} {n  \prod_{j=1}^{n-1}   (\sum_{h\in \mathcal{H}^t_j} \mu^2_h(j)) }\nonumber\\
&=&\dfrac{ \epsilon \sum_{q\in \mathcal{H}^t_{n-1}}  \nu^2_q(n-1)} {n  \prod_{j=1}^{n-2}   (\sum_{h\in \mathcal{H}^t_j} \mu^2_h(j)) },%=\cdots=\dfrac{\epsilon}{n}\sum_{i=1}^t a^2_{i,n}.
\end{eqnarray}
where   $ \nu_q(n-1)=\nu_{q,n}(n-1) $ for $q\in \mathcal{H}^t_{n-1}$.

Now, Lemma \ref{sumaa} implies that for any $k\in [1,n-2]$ and $h=(p,q)\in \mathcal{H}^t_{k+1}$,
\begin{eqnarray}\label{nuks}
&&\nu_{h,s}(k+1)=\zeta_{h,k+1}(a_{i,j},i=1,\ldots,t,j=1,\ldots,k,s)\nonumber\\
&=& \zeta_{p,k}(a_{i,j},i=1,\ldots,t,j=1,\ldots,k)\zeta_{q,k}(a_{i,j},i=1,\ldots,t,j=1,\ldots,k-1,s)\nonumber\\
&&-\zeta_{q,k}(a_{i,j},i=1,\ldots,t,j=1,\ldots,k)\zeta_{p,k}(a_{i,j},i=1,\ldots,t,j=1,\ldots,k-1,s)\nonumber\\
&=&\mu_p(k)\nu_{q, s}(k)-\mu_q(k)\nu_{p, s}(k),\qquad s=k+2,\ldots,n.
\end{eqnarray}
As a result, \dref{munue} yields
\begin{eqnarray*}
\sum_{h\in \mathcal{H}^t_{n-1}}  \nu^2_{h,n}(n-1)&=&\frac{1}{2}\sum_{p,q\in \mathcal{H}^t_{n-2}}( \mu_p(n-2)\nu_{q, n}(n-2)-\mu_q(n-2)\nu_{p, n}(n-2) )^2   \\
&\geq& \epsilon \left(\sum\nolimits_{p\in \mathcal{H}^t_{n-2}} \mu^2_p(n-2)\right) \left( \sum\nolimits_{q\in \mathcal{H}^t_{n-2}}   \nu^2_{q, n}(n-2)\right)\\
&\geq& \epsilon^{n-2}\prod_{j=1}^{n-2}  \left (\sum\nolimits_{h\in \mathcal{H}^t_j} \mu^2_h(j)\right) \left( \sum\nolimits_{q\in \mathcal{H}^t_{1}}   \nu^2_{q, n}(1)\right)\\
&=&\epsilon^{n-2}\prod_{j=1}^{n-2}  \left (\sum\nolimits_{h\in \mathcal{H}^t_j} \mu^2_h(j)\right)\left( \sum_{i=1}^t a^2_{i,n}  \right),
\end{eqnarray*}
which, by \dref{lambdana}, leads to
%\begin{eqnarray*}
$\lambda_n\geq \frac{\epsilon^{n-1}}{n} \left( \sum_{i=1}^t a^2_{i,n}  \right).$
%\end{eqnarray*}
 The lemma thus follows.
\end{proof}

\begin{lemma}\label{Sk}
Assume either
 $\underline{d}^m(\mathcal{P}'|\mathcal{S}^{2^{n-1}m})>1/C_w$ for $C_w<\infty$ or
 $\mathcal{P}' \neq \emptyset$ for $C_w=\infty$.
Then, the following two statements hold:\\
(i) there is a sequence of sets $\mathcal{B}_{jl}\triangleq \{b_{jl,s}    \}_{ s\in [1,N_{jl}]},  j\in [1,2^{n-1}], l\in[1,m] $ with integers $ N_{jl}\geq 1$   such that
 $\prod_{j=1  }^{2^{n-1}}\prod_{l=1}^m \mathcal{B}_{jl}\subset \mathcal{P}'   $,  and if $C_w<\infty$,
 \begin{eqnarray}\label{Sih}
%\left\{
%\begin{array}{ll}
\underline{d}(\mathcal{B}_{jl}| \mathcal{S})>1/C_w,  \,\,\, \forall  j\in [1,2^{n-1}], l\in [1,m];
%\\ N_{ij}\geq 1,& C_w=\infty
%\end{array}.
%\right.
\end{eqnarray}
(ii)  there is a number $ d>0$    such that
\begin{equation}\label{|gk|>0}
%C_g\triangleq
\min_{x\in\Theta^{2^{n-1}}} \min_{y\in \mathcal{D}} |g^n_n(x,y)|>0,
\end{equation}
where  $\mathcal{D} \triangleq  \prod\nolimits_{j=1}^{2^{n-1}}D_j$ and
 \begin{equation}\label{Dij}
 D_{j}\triangleq
  \prod\nolimits_{l=1}^m (\bigcup\nolimits_{ s\in [1, N_{jl}]}  [b_{jl,s}-d,b_{jl,s}+d]).%,\quad j\in [1,2^{n-1}].
\end{equation}
\end{lemma}

\begin{proof}
Since  either  $\underline{d}^m(\mathcal{P}'|\mathcal{S}^{2^{n-1}m})>1/C_w$ for $C_w<\infty$ or
 $\mathcal{P}' \neq \emptyset$ for $C_w=\infty$, statement (i) is straightforward ($ N_{jl}\equiv1$ for $C_w=\infty$).
To show statement (ii), note that   $g^n_n(x,y)$ is continuous, then for each $ x\in \Theta^{2^{n-1}}$, there is  a number $d_x>0$ and a neighbourhood $B_x$   of $x$  such that for
$D_{j}(x)\triangleq
  \prod\nolimits_{l=1}^m (\bigcup\nolimits_{ s\in [1, N_{jl}]}  [b_{jl,s}-d_x,b_{jl,s}+d_x]),$
$$\min\nolimits_{z\in B(x)}\min\nolimits_{y\in  \prod\nolimits_{j=1}^{2^{n-1}}D_j(x)} |g^n_n(z,y)|>0.$$
Now,  $\Theta^{2^{n-1}}$ is compact, by   the finite covering theorem, there is a  sequence  $\{x(i)\in \Theta^{2^{n-1}}\}_{i\in [1,N]}$ for some  $N\in \mathbb{N}^+$     such that
 $\Theta^{2^{n-1}}\subset \bigcup_{i\in [1,N]} B_{x(i)}$. So, \dref{|gk|>0} holds  by letting $d=\min_{1\leq i\leq N} d_{x(i)}$.
\end{proof}

To state the next lemma, denote   $\eta_t\triangleq\sum_{i=1}^t I_{i-m}$,  $\Omega_\eta\triangleq \left\{\omega: \lim_{t\rightarrow \infty} \eta_t=\infty \right\}$ and let $D_{j}, j\in [1,2^{n-1}]$ be defined by \dref{Dij}.

\begin{lemma}\label{yDeta}
Under the conditions of Theorem \ref{the}, for all  sufficiently large  $t$,
\begin{eqnarray}\label{varD}
\min\nolimits_{j\in [1,2^{n-1}] }(\sum\nolimits_{h=1}^t   I_{\{\varphi_h\in D_{j}\}}I_{\Omega_{h-m}})/\eta_{t}   >C_D\quad \mbox{a.s. on  }    \Omega_\eta,
\end{eqnarray}
where $C_D>0$ is a number independent of $t$.
\end{lemma}

\begin{proof}
Let filtration $\{\mathcal{F}_{h}\}$ be defined by \dref{Ft}.  Fix  $j\in [1,2^{n-1}]$.  Observe that for each   $l\in [0,m-1]$, $\{I_{\{\varphi_{hm+l}\in D_{j}\}}-P( \varphi_{hm+l}\in D_{j}|\mathcal{F}_{(h-1)m+l} ),\mathcal{F}_{hm+l}\}_{h\geq 0}$ is a  martingale difference sequence, then for all sufficiently large $t$,
\begin{eqnarray}\label{sumIvar}
&&\sum\nolimits_{h=1}^t   I_{\Omega_{h-m}} \left(I_{\{\varphi_h\in D_{j}\}}-P( \varphi_h\in D_{j}|\mathcal{F}_{h-m} )  \right)\nonumber\\
&  =&
o\left( \sum\nolimits_{h=1}^t I^2_{\Omega_{h-m}} \right)=o(\eta_t)\quad \mbox{a.s. on  }    \Omega_\eta.
\end{eqnarray}
For $h\geq m$, we compute $ P( \varphi_h\in D_{j}|\mathcal{F}_{h-m} ) I_{\Omega_{h-m}} $   by the following two cases:\\
(i) $C_w<\infty$. In this case,  $f_{h-l}=(f(\theta, \varphi_{h-l})+u_{h-l})$ falls in $\mathcal{S}$ for all $ l\in [1,m]$ on set $\Omega_{h-m}$. So,     \dref{Sih} yields
\begin{eqnarray}\label{fbCw}
 \max_{l\in [1,m]}  \min_{s\in [1, N_{jl}]}    \|b_{jl,s}-f_{h-l}\|<C_w,\quad   \mbox{ on } \Omega_{h-m}.
\end{eqnarray}
%\begin{eqnarray}\label{sumIvar}
%\lim_{t\rightarrow \infty}\dfrac{\sum_{s=m+1}^t \left(I_{\{\varphi_s\in D^k_{ihj}\}}-P( \varphi_s\in D^k_{ihj}|\mathcal{F}_{s-m} )  \right)   }{t}=0.
%\end{eqnarray}
For  $h\geq m$ and $l\in[1,m]$, denote
\begin{eqnarray*}
\Omega'_{h,l}&\triangleq&\{y_{h-l+1}\in\bigcup\nolimits_{ s\in [1, N_{jl}]} [b_{jl,s}-d,b_{jl,s}+d]\}\\
&=&\{ w_{h-l+1}\in \bigcup\nolimits_{ s\in [1, N_{jl}]}[b_{jl,s}-f_{h-l} -d,b_{jl,s}-f_{h-l}+d]\}.
\end{eqnarray*}
So, by Assumption B1 and \dref{fbCw},
there is a $C_d>0$ such that
\begin{eqnarray}\label{Cd}
E( I_{\Omega'_{h,l}}|\mathcal{F}_{h-l} )I_{\Omega_{h-m}}
%&=& P\left(w_{h-l+1}\in \bigcup_{ s\in [1, N_{i,jm+l}]}[b_{jl,s}(i)-f(\theta_t, \varphi_{h-l}) -d,b_{jl,s}(i)-f(\theta_t, \varphi_{h-l})+d]\Big|\mathcal{F}_{h-l}\right) I_{\Omega_{h-m}}\nonumber\\
 \geq \inf\nolimits_{z\in[-C_w, C_w]} P\{w_1\in (z-d,z]\}I_{\Omega_{h-m}}=     C_dI_{\Omega_{h-m}}
\end{eqnarray}
holds for all $h\geq m$ and  $l\in [1,m]$.   By virtue of \dref{Cd}, %by letting    $\bar \varphi_{s} \triangleq    \{ y_s,y_{s-1},\ldots, y_{s-2m+1}    \}$,
\begin{eqnarray}\label{CdI}
P( \varphi_h\in D_{j}|\mathcal{F}_{h-m} )I_{\Omega_{h-m}}
&=&E\left(\prod\nolimits_{l=1}^m I_{\Omega'_{h,l}}|\mathcal{F}_{h-m} \right)I_{\Omega_{h-m}}\nonumber\\
&=&E\left(  E( I_{\Omega'_{h,1}}|\mathcal{F}_{h-1} )I_{\Omega_{h-m}} \prod\nolimits_{l=2}^m I_{\Omega'_{h,l}}    |\mathcal{F}_{s-m}\right )\nonumber\\
   &\geq &C_d E\left(\prod\nolimits_{l=2}^m I_{\Omega'_{h,l}}    |\mathcal{F}_{s-m}\right )I_{\Omega_{h-m}} \nonumber\\
%%&\geq &  E(\cdots E(  E( I_{\{y_{s}\in [b^k_{hj1}(x^i)-d,b^k_{hj1}(x^i)+d]\}}|\mathcal{F}_{s-1} )I_{\{  \|\varphi_{s-1}\|\leq C_\varphi\}}\\   &&I_{\{y_{s-1}\in [b^k_{hj2}(x^i)-d,b^k_{hj2}(x^i)+d]|\mathcal{F}_{s-2} \}}I_{\{ \|\varphi_{s-2}\|\leq C_\varphi\}}  | \mathcal{F}_{s-2} ) \\
%&&\cdots    I_{\{y_{s-m+1}\in [b^k_{hjm}(x^i)-d,b^k_{hjm}(x^i)+d]|\mathcal{F}_{s-m} \}}I_{\{ \|\varphi_{s-m}\|\leq C_\varphi\}}  | \mathcal{F}_{s-m})\\
&\geq & \cdots\geq C^m_d I_{\Omega_{h-m}}.
\end{eqnarray}
%which is exactly \dref{varD}.\\
(ii) $C_w=\infty$. Note that %\bigcup_{i=1}^N\bigcup_{j=1}^{2^{n-1}}
$\{D_{j}\}_{ j\in [1,2^{n-1}]}$ are bounded, then there is a $C_f>0$ such that
$$
\min_{l\in[1,m]}\min_{s\in [1,N_{jl}]}\|b_{jl,s}-f_{h-l}\| I_{\{\bigcap\nolimits_{r=l+1}^m\Omega'_{h,r}\cap \Omega_{h-m}\}} \leq C_f.
$$
Since $N_{jl}\equiv 1$,   by Assumption B1,    for any $h\geq m$ and $l\in [1,m]$,
\begin{eqnarray*}
&&E( I_{\Omega'_{h,l}}|\mathcal{F}_{h-l} )I_{\Omega_{h-m}} \prod\nolimits_{s=l+1}^m I_{\Omega'_{h,s}}  \nonumber\\
 &\geq& \inf_{z\in[-C_f, C_f]} P\{w_1\in (z-d,z]\} I_{\Omega_{h-m}}\prod\nolimits_{s=l+1}^m I_{\Omega'_{h,s}}=     C_d I_{\Omega_{h-m}}\prod\nolimits_{s=l+1}^m I_{\Omega'_{h,s}},\quad \mbox{a.s.},
\end{eqnarray*}
where $C_d$ is a positive number.   So, \dref{CdI} also holds for this case.

Combined with \dref{sumIvar}, both the  two cases  indicate that  for all  sufficiently large $t$,
\begin{eqnarray}\label{sumIvarcons}
&&\dfrac{\sum_{h=1}^t I_{\{\varphi_h\in D_{j}\}} I_{\Omega_{h-m}} }{\eta_t } \nonumber\\
    &\geq & \dfrac{\sum_{h=1}^t    P( \varphi_h\in D_{j}|\mathcal{F}_{h-m} ) I_{\Omega_{h-m}}  }{\eta_t }-\dfrac{C_d^m}{2}
%&\geq&  \dfrac{C^m_d\sum_{s=m+1}^t   I_{\{\|\varphi_{s-m}\|\leq C_\varphi\}}  }{\sum_{s=m+1}^t I_{\{\|\varphi_{s-m}\|\leq C_\varphi\}} }-\dfrac{C_d^m}{2}\nonumber\\
 %  &\geq&     \dfrac{\sum_{s=m+1}^i  C'^m I_{\{ \sum_{l=1}^{m} |\triangle u_{s-l}| \leq C\}}}{i}
  % -\dfrac{C'^m\sigma}{2}\nonumber\\
   \geq \dfrac{C^m_d}{2}>0,\quad \mbox{a.s. on  }    \Omega_\eta.
\end{eqnarray}
Then, \dref{varD} follows from  \dref{sumIvarcons} by noting that $j$ is finite.
\end{proof}

Now, at time $t\geq 1$,  for any $k\in [1,n]$ and $ h=(h_1,h_2,\ldots,h_{2^{k-1}})\in \mathcal{H}^{t}_k $,   denote
$$
y^{(k)}_{h}\triangleq \mbox{col}\{\varphi_{h_1},\varphi_{h_2},\ldots, \varphi_{h_{2^{k-1}}}\}    \quad \mbox{with}\quad \varphi_{h_i}=(y_{h_i},\ldots,y_{h_i-m+1})^T.$$
%Let $\Omega'_h\triangleq\{\|\varphi_h\|\leq \gamma, \|\varphi_{h-m}\|\leq C\}, h\in \mathcal{H}^{t}_1$ and
Take $\gamma$  sufficiently large that for each $h\geq 1$,
\begin{eqnarray}\label{IC'}
\left\{
\begin{array}{ll}
I_{{\Omega_{h}}(\gamma,C)}=I_{\Omega_{h-m} },&C_w<\infty\\
I_{{\Omega_{h}}(\gamma,C)}\geq \max_{ j\in [1,2^{n-1}]}I_{\{ \varphi_{h} \in D_{j} \} }I_{\Omega_{h-m}},& C_w=\infty
\end{array}.
\right.
\end{eqnarray}
For $t\geq 1$, let $\{\theta_{t,h}\}_{ h\in [1,t]}$ be a sequence of random variables   taking values in $\Theta$ and
define $\vartheta_{t,h}, h\in \mathcal{H}^t_k, k\in [1,n]$ by $\vartheta_{t,h}\triangleq\mbox{col}\{\theta_{t,h_1},  \theta_{t,h_2}, \ldots, \theta_{t,h_{2^{k-1}}}  \}$.

\begin{lemma}\label{gks>}
Under the conditions of Theorem \ref{the}, there are some   $C_g, C_{g,\eta}>0$ such that  for all  $k\in [1,n],  s\in [k,n]$ and all sufficiently large $t$,
\begin{eqnarray}\label{numgk}
\sum\nolimits_{h\in  \mathcal{H}^{t}_k} I_{\{|g^k_s(\vartheta_{th},y^{(k)}_{h})|\geq C_g\}}\prod\nolimits_{j=1}^{2^{k-1}}  I_{\Omega_{h_j}(\gamma,C)}
\geq C_{g,\eta}\eta_{t}^{2^{k-1}}, \quad  \mbox{a.s. on  }   \Omega_\eta.\quad
\end{eqnarray}
\end{lemma}

\begin{proof} First,
in view of \dref{varD} and \dref{IC'}, for all sufficiently large $t$,
\begin{eqnarray*}
&&\sum\nolimits_{h\in  \mathcal{H}^{t}_n} I_{\{ y^{(n)}_{h} \in \mathcal{D} \} }\prod\nolimits_{j=1}^{2^{n-1}}  I_{\Omega_{h_j}(\gamma,C)}\\
&=& \sum\nolimits_{h\in  \mathcal{H}^{t}_n} \prod\nolimits_{j=1}^{2^{n-1}} I_{\{ \varphi_{h_j} \in D_j \} } I_{\Omega_{h_j}(\gamma,C)}\\
&\geq &\sum\nolimits_{h\in {\mathcal{H}}^{t}_n}\prod\nolimits_{j=1}^{2^{n-1}} I_{\{ \varphi_{h_j} \in D_j \} } I_{\Omega_{h_j-m}}
\geq  \frac{(C_D\eta_{t})^{2^{n-1}}}{2},\qquad \mbox{a.s. on  }   \Omega_\eta.
\end{eqnarray*}
Moreover, considering Lemma \ref{Sk}, let $C^g_{n}\triangleq\min_{x\in\Theta^{2^{n-1}}} \min_{y\in \mathcal{D}} |g^n_n(x,y)|>0,$ then
\begin{eqnarray}\label{numgn}
&&\sum\nolimits_{h\in  \mathcal{H}^{t}_n} I_{\{|g^n_n(\vartheta_{th},y^{(n)}_{h})|\geq C^g_{n}\}}\prod\nolimits_{j=1}^{2^{n-1}}  I_{\Omega_{h_j}(\gamma,C)} \nonumber\\
&\geq& \sum\nolimits_{h\in  \mathcal{H}^{t}_n} I_{\{ y^{(n)}_{h} \in \mathcal{D} \} }\prod\nolimits_{j=1}^{2^{n-1}}  I_{\Omega_{h_j}(\gamma,C)}
\geq \frac{(C_D\eta_{t})^{2^{n-1}}}{2},\qquad  \mbox{a.s. on  }   \Omega_\eta,
\end{eqnarray}
whenever $t$ is sufficiently large. % So, \dref{numgk} holds for $k=n$.

Now, recursively define
  $C^g_{k-1}\triangleq C^g_{k}/(2\bar C_g), k=n,\ldots,2$, where
 for $\overline { B(0,\gamma) } \subset \mathbb{R}^m$,
 $$\bar C_g\triangleq \max_{1\leq k\leq s\leq n} \max_{x\in \Theta^{2^{k-1}}}\max_{y\in \left(\overline { B(0,\gamma) }\right)^{2^{k-1}}  }|g^k_s(x,y)|.$$
% So, if $|g^n_n(\vartheta_{t,(p,q)},y^{(n)}_{p}, y^{(n)}_{q})|\geq C^g_{n}$
 Because of \dref{numgn},  suppose there is an integer  $k\in [2,n]$ such that for all    $s\in [k,n]$ and all sufficiently large $t$,
\begin{eqnarray}\label{ingk}
\sum\nolimits_{h\in  \mathcal{H}^{t}_k} I_{\{|g^k_s(\vartheta_{th},y^{(k)}_{h})|\geq C_g\}}\prod\nolimits_{j=1}^{2^{k-1}}  I_{\Omega_{h_j}(\gamma,C)}
\geq \frac{C_D^{2^{n-1}}}{ 2^{n+1-k}}\eta_{t}^{2^{k-1}},\quad   \mbox{a.s. on  }   \Omega_\eta.
\end{eqnarray}
Let  $s\in [k,n]$ and $h=(p,q)$ with   $p=(p_j)_{j=1}^{2^{k-2}}, q=(q_j)_{j=1}^{2^{k-2}}\in \mathcal{H}^t_{k-1}$. By \dref{gk}, on set $ (\bigcap\nolimits_{j=1}^{2^{k-2}}\Omega_{p_j}(\gamma,C))  \cap (\bigcap\nolimits_{j=1}^{2^{k-2}}\Omega_{q_j}(\gamma,C) )$, it is evident that for $r=k-1$ and $s$,
 \begin{eqnarray*}%\label{numgk}
|g^k_s(\vartheta_{th},y^{(n)}_{h})|\leq \bar C_g( |g^{k-1}_{r}(\vartheta_{tp},y^{({k-1})}_{p})|+|g^{k-1}_{r}(\vartheta_{tq},y^{({k-1})}_{q})| ).
\end{eqnarray*}
 As a result, both $r=k-1$ and $s$ lead to
\begin{eqnarray*}%\label{numgk}
&&\sum\nolimits_{h\in  \mathcal{H}^{t}_k} I_{\{|g^k_s(\vartheta_{th},y^{(k)}_{h})|\geq C^g_{k}\}}\prod\nolimits_{j=1}^{2^{k-1}}  I_{\Omega_{h_j}(\gamma,C)} \nonumber\\
&\leq&\sum\nolimits_{p,q\in  \mathcal{H}^{t}_{k-1}}\left (I_{\{|g^{k-1}_{r}(\vartheta_{tp},y^{({k-1})}_{p})|\geq C^g_{k-1}\}} + I_{\{|g^{k-1}_{r}(\vartheta_{tq},y^{({k-1})}_{q})|\geq C^g_{k-1}\}}\right) \\
&&\cdot\prod\nolimits_{j=1}^{2^{k-2}} I_{\Omega_{p_j}(\gamma,C)}\prod\nolimits_{j=1}^{2^{k-2}} I_{\Omega_{q_j}(\gamma,C)}\\
&\leq& 2\left( \sum\nolimits_{p\in  \mathcal{H}^{t}_{k-1}}I_{\{|g^{k-1}_{r}\vartheta_{tp},y^{({k-1})}_{p})|\geq C^g_{k-1}\}}\prod\nolimits_{j=1}^{2^{k-2}} I_{\Omega_{p_j}(\gamma,C)} \right )\left( \sum\nolimits_{q\in  \mathcal{H}^{t}_{k-1}} I_{\Omega_{q_j}(\gamma,C)} \right ),
\end{eqnarray*}
or equivalently,      by  \dref{ingk} and $\sum\nolimits_{q\in  \mathcal{H}^{t}_{k-1}} I_{\Omega_{q_j}(\gamma,C)} \leq \eta^{2^{k-2}}_t$,
\begin{eqnarray*}%\label{numgk}
&&  \sum\nolimits_{p\in  \mathcal{H}^{t}_{k-1}}I_{\{|g^{k-1}_{r}\vartheta_{tp},y^{({k-1})}_{p})|\geq C^g_{k-1}\}}\prod\nolimits_{j=1}^{2^{k-2}} I_{\Omega_{p_j}(\gamma,C)} \\
&\geq&\frac{\sum\nolimits_{h\in  \mathcal{H}^{t}_k} I_{\{|g^k_s(\vartheta_{th},y^{(k)}_{h})|\geq C^g_{k}\}}\prod\nolimits_{j=1}^{2^{k-1}}  I_{\Omega_{h_j}(\gamma,C)}}
{2\left( \sum\nolimits_{q\in  \mathcal{H}^{t}_{k-1}} I_{\Omega_{q_j}(\gamma,C)} \right ) }\\
&\geq & \frac{C^{2^{n-1}}_D\eta_{t}^{2^{k-1}}/2^{n+1-k}}
{2\eta^{2^{k-2}}_t }=\frac{C^{2^{n-1}}_D}{2^{n+2-k}}\eta^{2^{k-2}}_t, \qquad  \mbox{a.s. on  }   \Omega_\eta.
\end{eqnarray*}
This implies that \dref{ingk} is also true for $k-1$. The lemma is thus proved by taking $C_g=\min_{1\leq i\leq n}C^g_i$ and $C_{g,\eta}=\frac{C_D^{2^{n-1}}}{ 2^{n}}$.
\end{proof}

%\begin{lemma}
%Under Assumption A3, there is a set of squares $\mathcal{S}\triangleq \{ S_i \}$, where $S_i\triangleq \prod_{j=1}^{2^{k-1}}D_{ij}$ with $D_{ij}\triangleq  \prod_{l=1}^m [b_{ijl}-d,b_{ijl}+d]$ for some $d>0$, $1\leq i\leq \mbox{card}(\mathcal{S})<\infty$, such that

%\end{lemma}

%For $h=(h_1,h_2,\ldots,h_{2^{k-1}})\in \mathcal{H}^{t}_{k}$,
%denote $z^{(k)}_h\triangleq \mbox{col} \{z'_{h_1},z'_{h_2},\ldots,z'_{h_{2^{k-1}}}\}$ with $z'_{h_i}=(z_{h_i},z_{h_i-1},\ldots,z_{h_i-m+1})^T$ and $z_{h_i-j}\in \mathbb{R} $, $j\in [0,m-1],  i\in[1,2^{k-1}], k\in [1,n]$.

\begin{lemma}\label{sumuv>t}
For $t\geq 1$, let $\{\theta_{t,h}\}_{ h\in [1,t]}$ be a sequence of random variables   taking values in $\Theta$. In   Lemmas \ref{sumaa} and \ref{detsumaa},
set
\begin{equation}\label{alf}
\a_h=\frac{\partial f(x,\varphi_h)}{\partial x}\Big|_{x=\theta_{t,h}}I_{{\Omega_{h}}(\gamma,C)},\quad h\in [1,t],
\end{equation}
where $\gamma$ is a positive number. %For each $k\geq 2$, suppose  there are two constants $C,\sigma>0$  and two vectors $a^{(k-1)}, \bar a^{(k-1)}\in \mathbb{R}^{2^{k-1}}   $   such that
%\begin{eqnarray}\label{uC}
%&&\sum_{i=1}^t i^{2^{k-1}-1}I_{\{  |\triangle u_i| \leq C\}} \geq \sigma t^k\\
%&&\sum_{i=m+1}^t I_{\{ \sum_{l=1}^{m} |\triangle u_{i-l}| \leq C\}} \geq \sigma t\\
%\label{aagg}
%&&\inf_{x\in \Theta} g_{k-1}^{(k-1)}(x,a^{(k-1)}) g_{k}^{(k-1)}(x,\bar a^{(k-1)})>0.
%\end{eqnarray}
Then, under the conditions of Theorem \ref{the}, there is a $C_P>0$ independent of $t$ such that for all  be sufficiently large $t$,
\begin{equation}\label{sumu>t}
\sum\nolimits_{h=1}^t \nu^2_{h,s}(1)\geq C_P \eta_t, \,\, \forall s\in [1,n],   \quad     \mbox{a.s. on  }    \Omega_\eta.
\end{equation}
In addition, taking $\gamma$  appropriately large, there is a number $\epsilon>0$ such that for all  be sufficiently large $t$,   \dref{munue} holds a.s. on  $\Omega_\eta$ for each $k\in [1,n-1]$  and $s\in [k+1, n]$.
% there is  a $\delta>0$ such that
%$$
%\sum_{(p,q)\in  \mathcal{H}^{t}_k} \mu^2_{p}(k-1)\nu^2_{q}(k-1)\geq \delta \mbox{card}(\mathcal{H}^{t}_k)
%$$
\end{lemma}

\begin{proof}
%Define $\vartheta_{t,h}, h\in \mathcal{H}^t_k, k\in [1,n]$ by $\vartheta_{t,h}\triangleq\mbox{col}\{\theta_{t,h_1},  \theta_{t,h_2}, \ldots, \theta_{t,h_{2^{k-1}}}  \}$.
%By Lemma \ref{Sk},
%$$\min\nolimits_{z\in \mathcal{D}} |g^n_n(\vartheta_{t,h},z)|\geq C_g,\quad  h\in \mathcal{H}^t_n,$$
%where $C_g$ is defined by Lemma \ref{Sk}.
%So, when $t$ is sufficiently large,
% for each $k\in [1,n]$,
%\begin{eqnarray}
%&&\sum\nolimits_{h\in  \mathcal{H}^{t}_k} I_{\{|g^k_s(\vartheta_{t,h},y^{(k)}_{h})|\geq C_g\}}\prod\nolimits_{j=1}^{2^{k-1}}  I_{\Omega_{h_j}(\gamma,C)} \nonumber\\
%&\geq& \min\nolimits_{1\leq i\leq N}\sum\nolimits_{h\in  \mathcal{H}^{t}_k} I_{\{ y^{(k)}_{h} \in \mathcal{D}^k_{s}(i) \} }\prod\nolimits_{j=1}^{2^{k-1}}  I_{\Omega_{h_j}(\gamma,C)}\nonumber\\
%&\geq &\left(C_d^m\eta_{t}/2\right)^{2^{k-1}},\quad \forall s\in [k,n],\qquad \mbox{a.s. on  }   \Omega_\eta.
%\end{eqnarray}
First,
by \dref{gk}, \dref{muk+1}, \dref{nuks} and \dref{alf}, it is easy to verify that for each $k\in [1,n-1]$,   $h=(p,q)$, $p,q\in  \mathcal{H}^t_k$ with $p=(p_j)_{j=1}^{2^{k-1}}, q=(q_j)_{j=1}^{2^{k-1}} $ and $s\in[k+1,n]$,
\begin{eqnarray}\label{uvg}
\left\{
\begin{array}{ll}
\mu_p(k)=g_{k}^{k}(\vartheta_{tp}, y^{(k)}_{p})\prod\nolimits_{j=1}^{2^{k-1}}  I_{\Omega_{p_j}(\gamma,C)}\\
\nu_{q,s}(k)=g_{s}^{k}(\vartheta_{tq},  y^{(k)}_{q})\prod\nolimits_{j=1}^{2^{k-1}}  I_{\Omega_{q_j}(\gamma,C)}\\
g^{k+1}_{s}(\vartheta_{th},y^{(k+1)}_h)\prod\nolimits_{j=1}^{2^{k}}  I_{\Omega_{h_j}(\gamma,C)}=\mu_p(k)\nu_{q,s}(k)-\mu_q(k)\nu_{p,s}(k)
\end{array}.
\right.
\end{eqnarray}
As a consequence,  by Lemma \ref{gks>} and \dref{uvg}, for each $s\in [1,n]$,
\begin{eqnarray*}
 \sum\nolimits_{h\in [1,t]} \nu^2_{h,s}(1)
&= & \sum\nolimits_{h\in [1,t]}  ( g^{1}_{s}(\vartheta_{th},\varphi_h))^2I_{\Omega_{h}(\gamma,C)} \nonumber\\
&\geq&  C^2_g  \sum\nolimits_{h\in [1,t]}   I_{\{|g^1_s(\vartheta_{th},y^{(1)}_{h})|\geq C_g\}}I_{\Omega_{h}(\gamma,C)} \nonumber\\
&\geq& C^2_gC_{g,\eta}  \eta_{t},\qquad \mbox{a.s. on  }     \Omega_\eta.
\end{eqnarray*}
 Hence,  \dref{sumu>t} holds by letting $C_P=C^2_gC_{g,\eta} $. Furthermore, for each $k\in [1,n-1]$,
\begin{eqnarray}\label{uv-vu}
&& \sum\nolimits_{p,q\in \mathcal{H}^t_{k}} ( \mu_p(k)\nu_{q,s}(k)-\mu_q(k)\nu_{p,s}(k))^2\nonumber\\
&\geq& \sum\nolimits_{h\in  \mathcal{H}^{t}_{k+1}} ( g^{k+1}_{s}(\vartheta_{th},y^{(k+1)}_h))^2I_{\{  |g^{k+1}_{s}(\vartheta_{th},y^{(k+1)}_h)| \geq C_g\} } \prod\nolimits_{j=1}^{2^{k}}  I_{\Omega_{h_j}(\gamma,C)}\nonumber\\
%&\geq& C^2_g \sum\nolimits_{h\in  \mathcal{H}^{t}_{k+1}} I_{\{ y^{(k+1)}_{h} \in \mathcal{D}^{k+1}_s(i_t ) \} } \prod\nolimits_{j=1}^{2^{k}}  I_{\Omega_{h_j}(\gamma,C)}\nonumber\\
&\geq&  C^2_g C_{g,\eta}  \eta^{2^{k}}_{t},\qquad  \forall s\in [k+1,n], \qquad            \mbox{a.s. on  }    \Omega_\eta.
\end{eqnarray}
On the other hand,  for every $s \in [k+1,n]$,   % for $h=(p,q)\in \mathcal{H}^t_{k+1}$,
%similar to \dref{varD}--\dref{sumIvarcons}, one can deduce that
%$$
%\eta_{t-m}=O\left(\sum_{s=m+1}^t    I_{\{\|\varphi_s\|\leq C_\varphi\}}  I_{\{\|\varphi_{s-m}\|\leq C_\varphi\}} \right)\quad \mbox{a.s. on  }    %\left\{\lim_{t\rightarrow \infty} \eta_t=\infty \right\}.
%$$
 \begin{eqnarray*}
&&\sum\nolimits_{p,q\in \mathcal{H}^t_k}  \mu^2_p(k)\nu^2_{q,s}(k)\\
&=&\sum\nolimits_{p,q\in \mathcal{H}^t_k}(g_{k}^{(k)}(\vartheta_{tp},y^{(k)}_{p})g_{s}^{(k)}(\vartheta_{tq},y^{(k)}_{q}))^2\prod\nolimits_{j=1}^{2^{k-1}}  I_{\Omega_{p_j}(\gamma,C)} I_{\Omega_{q_j}(\gamma,C)} \\
&\leq&C'_g \sum\nolimits_{p,q\in \mathcal{H}^t_k}\prod\nolimits_{j=1}^{2^{k-1}}  I_{\Omega_{p_{j-m}}} I_{\Omega_{q_{j-m}}}\\
&=&  C'_g \sum\nolimits_{h\in \mathcal{H}^t_{k+1}}\prod\nolimits_{j=1}^{2^{k}}  I_{\Omega_{h_{j-m}}}   \leq C'_g \eta^{2^k}_{t},
\end{eqnarray*}
where % $p=(p_j)_{j=1}^{2^{k-1}}, q=(q_j)_{j=1}^{2^{k-1}}\in \mathcal{H}^t_k$ and
$$C'_g\triangleq \max_{k\in[1,n-1],s\in [k+1,n]}\max_{x\in \Theta^{2^{k-1}}}\max_{\max\{ \|z\|,\|z'\|\}\leq \sqrt{ 2^{k-1}  }\gamma }(g_{k}^{(k)}(x,z)g_{s}^{(k)}(x,z'))^2.$$
This with \dref{uv-vu}
 completes the proof by letting $\epsilon=(C^2_g C_{g,\eta} )/(2C'_g)$ in \dref{munue}.
\end{proof}

Let $t\geq 1$. For each $i\in  [1,t]$, denote $\phi_{i}(x)\triangleq \frac{\partial f(x,\varphi_i)}{\partial x}$ and  define
\begin{eqnarray*}
\left\{
\begin{array}{ll}
P^{-1}_{t+1}(x)\triangleq \sum_{i=1}^{t}  \phi_{i}(x_i) \phi^T_{i}(x_i)I_{\Omega_{i}(\gamma,C)}\\
r_t\triangleq\sum_{i=1}^t \max_{x\in\Theta} \|\phi_i(x)\|^2 I_{\Omega_{i}(\gamma,C)}
\end{array},
\right.
\end{eqnarray*}
where  $x=\mbox{col}\{x_{1},\ldots,x_{t}\},x_i\in\Theta$. % and   $\gamma$ is set in Lemma \ref{sumuv>t}.
The next lemma is straightforward.% The following lemma shows
\begin{lemma}\label{minlambda}
Under the conditions of Theorem \ref{the}, let $\theta_{t}\triangleq\mbox{col}\{\theta_{t,1},\ldots,\theta_{t,t}\}$, where    $\{\theta_{t,h}, h\in [1,t]\}$ is a sequence of random variables   taking values in $\Theta$. Then,\\
(i) for all sufficiently large $t$ and
$\gamma$,     there is a random positive  number
 $C_1$ such that
 \begin{eqnarray}\label{PC1}
\lambda_{\min}\left(P^{-1}_{t+1}(\theta_t) \right)\geq  C_1 \eta_t,\quad \mbox{a.s. on  }    \Omega_\eta;
\end{eqnarray}
 (ii)
let $C_2=\max_{x\in \Theta,\|z\|\leq \gamma}\| \frac{\partial f(x,z)}{\partial x}\|^2$ with $\gamma>0$ sufficiently large, then
\begin{eqnarray}\label{rC2}
r_t\leq C_2 \eta_t,\quad \forall t>1.
\end{eqnarray}
\end{lemma}

\begin{proof}
Note that $P^{-1}_{t+1}(\theta_t)\geq \sum_{i=1}^t \a_i\a^T_i$, where $\a_i$ is defined by \dref{alf}.  In view of Lemma \ref{sumuv>t}, there is a number $\epsilon>0$ such that \dref{munue} holds  almost surely on  $\Omega_\eta$ for each $k\in [1,n-1]$ and $s\in [k+1,n]$, and hence Lemma \ref{detsumaa} yields
\begin{eqnarray*}
\lambda_{\min}\left(P^{-1}_{t+1}(\theta_t)\right) &\geq& (\epsilon^{n-1}/n)  \min\nolimits_{ s\in [1,n]}( \sum\nolimits_{h=1}^t  \nu^2_{h,s}(1)) \\
&\geq &\epsilon^{n-1}C_p\eta_t/n,\qquad \mbox{a.s. on  }    \Omega_\eta,
\end{eqnarray*}
where \dref{PC1} follows directly from \dref{sumu>t} in Lemma \ref{sumuv>t}.

Next, we show \dref{rC2}.
By \dref{Omegabar}, if  $C_w=\infty $,  it is clear that
$$
r_t\leq  C_2  \sum\nolimits_{i=1}^t  I_{\Omega_{i}(\gamma,C)}\leq  C_2  \sum\nolimits_{i=1}^t  I_{\Omega_{i-m}} = C_2   \eta_t.
$$
%where $C_2=\max_{x\in \Theta,\|z\|\leq \gamma}\| \frac{\partial f(x,z)}{\partial x}\|$.
When $C_w<\infty$, without loss of generality, assume  $\sup_{i\geq 1}\|\varphi_i\|I_{ \Omega_{i-m} }<\gamma$. Hence \dref{rC2} follows as well.
\end{proof}

\begin{lemma}\label{sumw-sig}
Under  Assumption B1, for any  $ \varepsilon\in (0,1)$,
\begin{eqnarray}\label{sumIw<}
%\left\{
%\begin{array}{l}
%I_{\Omega_t} |w_{t+1}|=O( \sqrt{\eta_t})\\
\left|\sum\nolimits_{i=1}^{t}I_{\Omega_i(\gamma,C)}(  w^2_{i+1}-\sigma^2_w)\right|
\leq  (1+\varepsilon)\bar{\sigma}_w  \sqrt{2\eta_{t}\log \log\eta_{t}}
%\end{array}
%\right.
,\quad \mbox{a.s. on  }    \Omega_\eta,
\end{eqnarray}
%whenever $t$ is  sufficiently large.  Moreover, for any $\tau\in (2,\sqrt{\kappa}]$,
%\begin{eqnarray}\label{phiwr}
%\max_{x\in\Theta} \|\phi_t(x)\|   I_{\Omega_t(\gamma,C)} |w_{t+1}|=o\left(  \left(r_t(\tau)\right)^{\frac{\tau+2}{4\tau} %}\right)+O(1),\quad \mbox{as}\,\, t\rightarrow\infty\quad \mbox{a.s.}.
%\end{eqnarray}
\end{lemma}

\begin{proof}
By Assumption B1,   $m'\triangleq E|w_1|^{\tau}$ exists for any $\tau\in (2,\sqrt{\kappa}]$.  Observe that $\tau^2\in(1,\kappa]$, employing the Minkowski inequality and the Lyapunov inequality yields
\begin{eqnarray*}
E||w_1|^{\tau}-m'|^{\tau}\leq\left(E|w_1|^{\tau^2}\right)^{\frac{1}{\tau}}+m'
\leq  \left( \left(E|w_1|^{\kappa}\right)^{\frac{\tau}{\kappa}}+m'     \right)^{\tau}<\infty.
\end{eqnarray*}
Since $\{|w_{i}|^{\tau}-m',\mathcal{F}_i\}$ is a martingale difference sequence with
 $$\sup_{i\geq 1} E(||w_{i+1}|^{\tau}-m'|^{\tau}|\mathcal{F}_i)<\infty,$$
  \cite[Lemma 2(iii)]{lw82} shows that
$$
\sum\nolimits_{i=1}^{t} I_{\Omega_i(\gamma,C)} (|w_{i+1}|^{\tau}-m')^2=O\left(  \eta_t \right),\quad \mbox{a.s. on  }    \Omega_\eta,
$$
and hence, as $t\rightarrow\infty$,     %for all sufficiently large $t$,
\begin{equation}\label{|w|<}
I_{\Omega_t(\gamma,C) }w^2_{t+1}=O(  \sqrt[\tau]{\eta}),\quad \mbox{a.s. on  }    \Omega_\eta.
\end{equation}
%$$
%E(I_{\Omega_i}(w^2_{i+1}-\sigma^2_w)|\mathcal{F}_i)=I_{\Omega_i}\bar{\sigma}^2_{w}.
%$$
Note that
$$\sum\nolimits_{i=1}^{t}  E(I_{\Omega_i(\gamma,C)}(w^2_{i+1}-\sigma^2_w)^2|\mathcal{F}_i)=\bar{\sigma}^2_{w}\eta_t\rightarrow\infty,\quad \mbox{a.s. on  } \Omega_\eta,$$
which, together with \dref{|w|<} and $\tau>2$, implies that as $i\rightarrow\infty$,
$$
\dfrac{I_{\Omega_i(\gamma,C)} |w^2_{i+1}-\sigma^2_w| \sqrt{2\log\log (\bar{\sigma}^2_{w}\eta_i)}}{\bar{\sigma}_{w}\sqrt{\eta_i}}=O\left( \eta^{-(\frac{1}{2}-\frac{1}{\tau})}_i  \sqrt{\log \log\eta_{i}}  \right)\rightarrow 0   ,\quad \mbox{a.s. on  } \Omega_\eta.
$$
Applying \cite[Corollary 5.4.2]{stout} to the martingale difference sequence $\{I_{\Omega_{i-1}(\gamma,C)}(  w^2_{i}-\sigma^2_w), \mathcal{F}_{i}\}$ yields
$$
\limsup_{t\to \infty} \dfrac{\left|\sum_{i=1}^{t}I_{\Omega_i(\gamma,C)}(  w^2_{i+1}-\sigma^2_w)\right|}{\sqrt{2\bar{\sigma}^2_{w}\eta_t\log\log (\bar{\sigma}^2_{w}\eta_t)}}\leq 1,\quad \mbox{a.s. on  } \Omega_\eta,
$$
which leads to \dref{sumIw<} immediately.
\end{proof}

\begin{lemma}\label{Gthe'}
Under Assumption B1, if  \dref{PC1} and \dref{rC2} hold for every $\theta_t$ defined in Lemma \ref{minlambda}, then for any $\lambda\in (0,\frac{1}{4}-\frac{1}{2\kappa})$,
\begin{eqnarray*}
\|\tilde{\theta}_{t}\|=O\left(\dfrac{1}{t^{\frac{1}{4}-\frac{1}{2\kappa}-\lambda}}\right)\rightarrow 0,\quad \mbox{a.s. on  }    \Omega'_\eta,
\end{eqnarray*}
where $\Omega'_\eta\triangleq\{t/\eta_t=O(1)\}$.
\end{lemma}

\begin{proof}
Fix $\lambda\in (0,\frac{1}{4}-\frac{1}{2\kappa})$  in the algorithm and let $\theta$  be the true value of the parameter.  It is clear that for every sufficiently large $t$,
$
(\theta,\sigma^2_w)\in  \Theta \times  \Sigma_t.
$
Therefore, there are two  points $o_{ti_t}\in \Theta$ and $\sigma^2_{tj_t}\in \Sigma_t$ such that
\begin{eqnarray}\label{osigma'}
\|\theta-o_{ti_t}\|\leq  \frac{n^{\frac{1}{2}}}{2t^{\frac{1}{4}-\frac{1}{2\kappa}-\lambda}}\,\,\, \mbox{and}\,\,\, 0\leq \sigma^2_{tj_t}-  \sigma^2_w< \frac{1}{t^{\frac{1}{2}-\frac{1}{\kappa}-\lambda}}.
\end{eqnarray}
Since $\Theta$ is convex,   \dref{hatGt=0} shows
\begin{eqnarray}\label{Gtrho}
\hat G_{t}(o_{ti_t},\sigma^2_{tj_t})&=&  \sum\nolimits_{s=1}^{t-1} (f(o_{ti_t},  \varphi_{s})-f(\theta, \varphi_{s})-w_{s+1})^2 I_{\Omega_s(\gamma,C)}-\eta_{t-1}(\gamma)\sigma^2_{tj_t}\nonumber\\
 &=&\sum\nolimits_{s=1}^{t-1} (\phi^T_{s}(\theta_{ts})(o_{ti_t}-\theta))^2I_{\Omega_s(\gamma,C)}\nonumber\\
 &&-\sum\nolimits_{s=1}^{t-1}(2(f(o_{ti_t}, \varphi_{s})-f(\theta,  \varphi_{s}))w_{s+1}-w^2_{s+1})I_{\Omega_s(\gamma,C)}  -\eta_{t-1}(\gamma)\sigma^2_{tj_t} \nonumber\\
 &=& (o_{ti_t}-\theta)^T P^{-1}_{t}(\theta_{t}) (o_{ti_t}-\theta) \nonumber\\
 &&-\sum\nolimits_{s=1}^{t-1} (2(f(o_{ti_t}, \varphi_{s})-f(\theta,  \varphi_{s}))w_{s+1}-w^2_{s+1})I_{\Omega_s(\gamma,C)}-\eta_{t-1}(\gamma)\sigma^2_{tj_t},
    \end{eqnarray}
where $\theta_{ts}, s\in [1,t]$ are $t$ random variables taking values in $ \Theta$  and $\theta_{t}=\mbox{col}\{\theta_{t1},\ldots,\theta_{tt}\}.$

 We now estimate the three terms in \dref{Gtrho}. First, by \dref{rC2}, \dref{osigma'} and $\eta_t\leq t$,
 \begin{eqnarray}\label{Po2}
 (o_{ti_t}-\theta)^T P^{-1}_{t}(\theta_{t}) (o_{ti_t}-\theta)&\leq &r_{t-1}\|o_{ti_t}-\theta\|^2\nonumber\\
 &\leq& \frac{C_2\eta_t n}{4t^{\frac{1}{2}-\frac{1}{\kappa}-2\lambda}}\leq \frac{C_2 n}{4}t^{\frac{1}{2}+\frac{1}{\kappa}+2\lambda}.
   \end{eqnarray}
Next, for each  $t\geq 1$, denote the number of points $o_{ti}$  in the algorithm by $N_o(t)$. Evidently, $N_o(t)=O(t^n)$. We estimate the following sequences
$$\sum\nolimits_{s=1}^{t-1} 2(f(o_{ti}, \varphi_{s})-f(\theta,  \varphi_{s}))I_{\Omega_s(\gamma,C)} w_{s+1},\quad i=1,2,\ldots, N_o(t), \,\, t=1,2,\ldots$$
To this end, note that by Assumption B1 and the Borel-Cantelli  theorem,
$$|w_s|=o(s^{\frac{1}{\kappa}+\lambda}),\quad s\rightarrow\infty\quad \mbox{a.s.}.$$
Moreover, $2(f(o_{ti}, \varphi_{s})-f(\theta,  \varphi_{s}))I_{\Omega_s(\gamma,C)}$ is $\mathcal{F}_s$-measurable for all $o_{ti}, i\in [1,N_o(t)], t\geq 1$.  Then, in view of  \cite[Lemma 3.2]{huangguo90}, it yields
\begin{eqnarray}\label{maxsumffw}
&&\max_{1\leq i\leq N_o(t)}|\sum\nolimits_{s=1}^{t-1} 2(f(o_{ti}, \varphi_{s})-f(\theta,  \varphi_{s}))I_{\Omega_s(\gamma,C)} w_{s+1}|\nonumber\\
&=& O(\sqrt{\eta_t}\log \eta_t)+o(\sqrt{\eta_t}t^{\frac{1}{\kappa}+\lambda}\log t)=o(t^{\frac{1}{2}+\frac{1}{\kappa}+\lambda}\log t),\quad \mbox{a.s. on } \Omega_\eta.
\end{eqnarray}
and hence
\begin{eqnarray}\label{sumffw}
\qquad |\sum\nolimits_{s=1}^{t-1} 2(f(o_{ti_t}, \varphi_{s})-f(\theta,  \varphi_{s}))I_{\Omega_s(\gamma,C)} w_{s+1}|= o(t^{\frac{1}{2}+\frac{1}{\kappa}+\lambda}\log t),\quad \mbox{a.s. on } \Omega_\eta.
\end{eqnarray}
At last,
 for all sufficiently large $t$,  Lemma \ref{sumw-sig} implies
\begin{eqnarray}\label{w2rhosi}
\qquad && |\sum\nolimits_{s=1}^{t-1}I_{\Omega_s(\gamma,C)} w^2_{s+1}-\sigma^2_w\eta_{t-1}(\gamma)| \leq 2 \bar\sigma_{w}  \sqrt{2\eta_{t}\log \log\eta_{t}}=o( t^{\frac{1}{2}+\lambda}),\quad \mbox{a.s. on  }    \Omega_\eta,
    \end{eqnarray}
which, together with \dref{osigma'}, yields
\begin{eqnarray}\label{w2rho}
&& |\sum\nolimits_{s=1}^{t-1}I_{\Omega_s(\gamma,C)} w^2_{s+1}-\sigma^2_{tj_t}\eta_{t-1}(\gamma)| =o( t^{\frac{1}{2}+\frac{1}{\kappa}+2\lambda}),\quad \mbox{a.s. on  }    \Omega_\eta,
    \end{eqnarray}
So, combining \dref{Po2}, \dref{sumffw} and \dref{w2rho}, for all sufficiently large $t$,
\begin{eqnarray*}
|\hat G_{t}(o_{ti_t},\sigma^2_{tj_t})|\leq (C_2 n/4+1)t^{\frac{1}{2}+\frac{1}{\kappa}+2\lambda},\quad \mbox{a.s. on  }    \Omega_\eta.
\end{eqnarray*}
Note that $C_\phi=C_2 n/4+1$, then $\mathcal{J}_{t}\neq \emptyset$. By \dref{ij*} and \dref{osigma'}, for every sufficiently large $t$, there is a point $(o_{ti^*},\sigma^2_{tj^*})\in \mathcal{J}_{t}$ satisfying
\begin{eqnarray}\label{tj*}
\sigma^2_{tj^*}-  \sigma^2_w<\frac{1}{t^{\frac{1}{2}-\frac{1}{\kappa}-\lambda}}.
\end{eqnarray}

We claim that $\|o_{ti^*}-\theta\|^2=O\left(\frac{1}{t^{\frac{1}{2}-\frac{1}{\kappa}-2\lambda} }\right)$ on $\Omega'_\eta$ almost surely. Otherwise, there
is a set $\Omega''_\eta\subset \Omega'_\eta$ with $P(\Omega''_\eta)>0$ such that
$$
\limsup_{t\to \infty}\|o_{ti^*}-\theta\|^2 t^{\frac{1}{2}-\frac{1}{\kappa}-2\lambda}=\infty,\quad \mbox{on }\Omega''_\eta.
$$
By \dref{PC1}, \dref{maxsumffw}, \dref{w2rhosi}, \dref{tj*} and the fact $\eta_{t-1}(\gamma)
\leq t$, a random $\theta_{t,ti^*}$ exists that for all sufficiently large $t$,
\begin{eqnarray*}
&&\limsup_{t\to \infty}\frac{|\hat G_{t}(o_{ti^*},\sigma^2_{tj^*})|}{ t^{\frac{1}{2}+\frac{1}{\kappa}+2\lambda} }\\
&\geq & \limsup_{t\to \infty} \left(  \frac{ \lambda_{\min}(P^{-1}_{t}(\theta_{t,ti^*}))\|o_{ti^*}-\theta\|^2}{ t^{\frac{1}{2}+\frac{1}{\kappa}+2\lambda} }+ \dfrac{(\sigma^2_w-  \sigma^2_{tj^*})\eta_{t-1}(\gamma)}{ t^{\frac{1}{2}+\frac{1}{\kappa}+2\lambda} }-o(1)\right)\\
&\geq &\limsup_{t\to \infty} \left(  C_1 \|o_{ti^*}-\theta\|^2t^{\frac{1}{2}-\frac{1}{\kappa}-2\lambda}-o(1)\right)=\infty,\quad \mbox{a.s. on  }    \Omega''_\eta,
\end{eqnarray*}
which contradicts to the  fact that $(i^*,j^*)\in \mathcal{J}_{t}$. Consequently, by \dref{the'},
$$
\|\hat\theta_{t}-\theta\|^2=\|o_{ti^*}-\theta\|^2=O\left(\dfrac{1}{t^{\frac{1}{2}-\frac{1}{\kappa}-2\lambda}}\right),\quad \mbox{a.s. on  }    \Omega'_\eta,
$$
as desired in \dref{the'err}.
\end{proof}

\emph{Proof of Theorem \ref{the}:}  When the closed-loop system is stable, $P(\Omega'_\eta)=1$ and  hence the theorem is a  direct result of Lemmas \ref{sumaa}--\ref{Gthe'}.

%\numberwithin{lemma}{section}
%\appendix

%\section{Proofs of Observation \ref{NS}}\label{AppB}


\begin{thebibliography}{99}



\bibitem{AW71}
 K. J. Astrom and B. Wittemnark,  \emph{Problems of identification
and control},  J. Math. Anal.  Applic.,  34 (1971), pp. 90--113.



%\bibitem{cp05}
%A. Chiuso and  G. Picci, ``Consistency analysis of some closed-loop subspace identification methods'',
% \emph{Automatica}, vol. 41, pp. 277--391, 2005.



\bibitem{FL99}
U. Forssell and L. Ljung, \emph{Closed-loop identification revisited}, Automatica, 35 (1999), pp. 1215--1241.




\bibitem{gevers86}
M. Gevers and  L. Ljung,  \emph{Optimal experiment  design with respect to the intended model application}, Automatica, 22 (1986), pp. 543--554.



\bibitem{GM1978}
P. E. Gill and  W. Murray, \emph{Algorithms for the solution of the nonlinear least-squares problem}, SIAM J. Numer. Anal., 15 (1978), pp. 977--992.





\bibitem{GLT}
N. I. M. Gould, S. Leyffer, and P. L. Toint, \emph{A  multidimensional filter algorithm for nonlinear
equations and nonlinear least-squares}, SIAM J. Optim., 15 (2004), pp. 17--38.




\bibitem{GBP74}
 M. S. Grewal, G. A. Bekey, and H. J. Payne, \emph{Parameter identifiability of dynamcal system}, in Proc. IEEE Conf. Decision  and Conrrol, 1974, pp. 446--448.




\bibitem{glover76}
M. S. Grewal and K. Glover, \emph{Identifiability of linear and nonlinear dynamical systems}, IEEE
Trans. Automat. Control, 21 (1976), pp. 833--837.






\bibitem{huangguo90}
D. W. Huang and L. Guo, ``Estimation of nonstationary ARMAX models based on Hannan-Rissanen method'', \emph{The Annals of Statistics}, 18 (1990),  pp. 1729--1756.






\bibitem{lw82}
T. L. Lai and C. Z. Wei, ``Least squares estimates in stochastic regression models with applications to identification and control of dynamic systems,'' \emph{The Annals of Statistics}, 10 (1982), pp. 154--166.


\bibitem{GLSlover77}
I. Gustavsson, L. Ljung, and T. S$\ddot{o}$derstr$\ddot{o}$m,   \emph{Identification of processes in closed loop---identifiability and accuracy aspects}, Automatica, 13 (1977), pp. 59--75.


\bibitem{HS95}
P. M. J. Van den Hof and R. J. P. Schrama,
\emph{Identification and control---closed-loop issues}, Automatica, 31 (1995), pp. 1751--1779.


\bibitem{Jabob}
C. Jacob,
\emph{Conditional least squares estimation in nonstationary nonlinear stochastic regression models}, Ann. Statist, 38 (2010), pp. 566--597.




\bibitem{lai}
T. L. Lai, \emph{Asymptotic properties of nonlinear least squares estimates in stochastic regression models}, Ann. Statist, 22 (1994), pp. 1917--1930.



\bibitem{lichen16}
C. Li and M. Z. Q. Chen, \emph{Simultaneous identification and stabilization  of nonlinearly  parameterized discrete-time systems  by nonlinear least squares algorithm}, IEEE
Trans. Automat. Control,    61 (2016),  pp. 1810--1823.




\bibitem{stout}
W. F. Stout,
\emph{Almost sure convergence}, Academic Press,  1974.



\bibitem{Wu81}
C. F. Wu,  \emph{Asymptotic theory of nonlinear least squares estimation}, Ann. Statist. 9  (1981), pp. 501--513.

\end{thebibliography}
\end{document}